\documentclass[11pt,reqno]{amsart}
\usepackage{amssymb}
\usepackage{amsmath}
\usepackage{amsfonts}
\usepackage{epsfig}
\usepackage{color}
\usepackage{epstopdf}
\usepackage{graphicx}
\usepackage{amsthm}
\usepackage{enumerate}
\usepackage[mathscr]{eucal}
\usepackage{verbatim}
\usepackage{multirow}
\usepackage{pst-grad}
\usepackage{pstricks}
\usepackage{pst-node}
\usepackage{cases}

  \theoremstyle{plain}
\newtheorem{theorem}{Theorem}[section]
\newtheorem{lemma}[theorem]{Lemma}
\newtheorem{lem}[theorem]{Lemma}
\newtheorem{proposition}[theorem]{Proposition}
\newtheorem{corollary}[theorem]{Corollary}

\newtheorem{rmk}[theorem]{Remark}

\theoremstyle{plain}

\newcommand{\cuxi}{\mathcal I}

\DeclareMathOperator{\dist}{dist}
\DeclareMathOperator{\supp}{supp}

\newcommand{\ef}[1]{f_{\cuxi^{#1}}}
\newcommand{\eff}[2]{f_{\cuxi^{#1}_{#2}}}
\newcommand{\cu}[1]{\cuxi^{#1}}
\newcommand{\cuu}[2]{\cuxi^{#2}_{#1}}

\newcommand{\gth}{\gamma_\tau^h}

\newcommand{\clag}{\Gamma (\epsilon)}

\title[Averaged decay estimates over curves]
{Averaged decay estimates \\ for Fourier transforms of measures \\ over curves with nonvanishing  torsion}
\author{Yutae Choi}
\author{Seheon Ham}
\author{Sanghyuk Lee}
\thanks{S. Lee is supported in part by NRF-2015R1A2A2A05000956 (Republic of Korea).}
\address{Department of Mathematics, Pohang University of Science and Technology, Pohang 37673, Republic of Korea}
\email{yutae\_choi@postech.ac.kr}
\address{School of Mathematics, Korea Institute for Advanced Study, Seoul 02455, Republic of Korea}
\email{hamsh@kias.re.kr}
\address{Department of Mathematical Sciences, Seoul National University, Seoul 08826, Republic of Korea}
\email{shklee@snu.ac.kr}

\begin{document}

\begin{abstract}
We study averaged decay estimates for Fourier transforms of measures
when the averages are taken over space curves with non-vanishing
torsion. We extend the previously known results to higher dimensions and discuss sharpness of the estimates. 
\end{abstract}

\maketitle

\section{Introduction}
Let $\mu$ be a positive Borel measure with compact support in
$\mathbb R^d$.
For $0 < \alpha < d$, the $\alpha$-dimensional energy of $\mu$ is
given by
    \begin{align*}
    I_{\alpha}(\mu) = \iint |x-y|^{-\alpha} d\mu(x) d\mu(y).
    \end{align*}
The energy  $I_{\alpha}$ has been widely used in various studies,
especially geometric measure theory problems, to describe regularity property of measure. In fact, it is well
known that finiteness of energy determines the Hausdorff dimension
of the support of $\mu$.
Finiteness of
$I_{\alpha}(\mu)$ and $L^{2}$ averaged decay estimates of
$\widehat{\mu}$ over the ball $B(0,1)$ are closely related. Here $B(x,r)$ denotes the
ball  which is centered at $x$ and of radius $r$. Indeed,
by the identity
    \begin{align*}
    \iint |x-y|^{-\alpha} d\mu(x) d\mu(y)
    = C_{\alpha,d} \int |\widehat{\mu}(\xi)|^{2} |\xi|^{\alpha - d} d\xi
    \end{align*}
it follows that $I_{\alpha}(\mu) < \infty$ for $\alpha < \delta$
provided that
    $
    \int_{B(0,1)} |\widehat{\mu}(\lambda\xi)|^{2} d\xi \leq C \lambda^{-\delta}
    $
for a positive constant $\delta$. Conversely, if  $I_{\alpha}(\mu) <
\infty$, it follows that
    $
    \int_{B(0,1)} |\widehat{\mu}(\lambda\xi)|^{2} d\xi
      \lesssim \lambda^{-\alpha} I_{\alpha}(\mu).
    $
(See Chapter $8$ in \cite{Wo2} and Chapter $12$ in \cite{Mab} for
further details.)

If $B(0,1)$ is replaced by a smooth submanifold of lower dimension,
it is expected that the decay rate gets worse. In connection with
problems in geometric measure theory there have been attempts to
characterize averaged decay over smooth manifolds. As is well
understood in  problems such as Fourier restriction problems, the
curvature properties of the underlying  submanifolds become important.

Let $\Sigma$ be a smooth compact submanifold with measure $d\nu$.
Let us consider the estimate, for $\lambda>1$,
 \begin{align}\label{eqn:submanifold}
   \int_{\Sigma} |\widehat{\mu}(\lambda\xi)|^{2} d\nu(\xi) \leq C \lambda^{-\zeta }
    I_{\alpha}(\mu).
    \end{align}
    In addition to $I_\alpha(\mu)<\infty$  the estimate
\eqref{eqn:submanifold} has been studied under the assumption that
\[ |\widehat{\nu}(\xi)| \lesssim |\xi|^{-a}, \quad \nu(B(x,\rho))
\lesssim \rho^{b}.\] The following can be found in \cite{Er}: If
$0<a,b<d$ and a compactly supported probability measure $\nu$
satisfies the above condition, then
  \eqref{eqn:submanifold} holds with  $\zeta =\max(\min(\alpha,a),\, \alpha -d +b)$.

In particular, in relation to the Falconer distance set problem (cf.
\cite{Mab,Wo2,Er}) the case that $\Sigma$ is the unit sphere
$\mathbb S^{d-1}$ and $\nu$ is the usual surface measure was studied
extensively after Mattila's contribution \cite{Ma} to  the  Falconer distance set problem. An extension of Mattila's estimate in \cite{Ma} was later obtained by Sj\"olin \cite{Sj}.  
The results in \cite{Ma,Sj} were based on a rather straightforward $L^2$
argument. Their results were further improved subsequently by
Bourgain, Wolff and Erdo\u{g}an \cite{Bo, Wo1, Er1, Er2}. These improvements were based on sophisticated  methods which were developed in the study of the Fourier restriction
problem (and Bochner-Riesz conjecture).
Especially in $\mathbb R^2$, for $\Sigma=\mathbb S^1$ the sharp estimates were
established by Mattila \cite{Ma} and Wolff \cite{Wo1}. (See also
Erdo\u{g}an \cite{Er, Er1, Er2}.) 
{In fact, it is proved in \cite{Ma,Sj}
that \eqref{eqn:submanifold} holds with $\zeta \leq \max( \min(\alpha,1/2),\alpha-1)$ and $\zeta$ should be smaller than or equal to $\max(\min(\alpha,1/2),\alpha/2)$.  
Later Wolff proved that \eqref{eqn:submanifold} holds with $\zeta < \alpha/2$ for $0< \alpha <2$.}
Recently a related result was obtained
by replacing the circle with a certain class of general curves in
$\mathbb R^2$ by Erdo\u{g}an and Oberlin \cite{ErOb}.

In this paper, we are concerned with the average of $\widehat \mu$ over space curves in
$\mathbb{R}^{d}$, $d \geq 3$.  Let $\gamma : I = [0,1] \to
\mathbb{R}^{d}$ be of a $C^{d+1}$ curve satisfying
    \begin{align}\label{eqn:non-degenerate}
    \det(\gamma'(t), \gamma''(t), \cdots, \gamma^{(d)}(t)) \neq 0 \textrm{ for }\, t \in I.
    \end{align}
As is to seen later, the averaged estimate over curves are closely
related to the restriction estimates for the curves which have been studied by
various authors. We refer the reader to  \cite{BL, BOS2, BOS1, BOS3,
DeMu, DeWr, Dr, DrMa, DrMa1, HaLe, Sto} and references therein.

 For a nonnegative number $x$, let us denote by $ [x],
\langle x\rangle$  the integer part and the fractional part of $x$,
respectively. The following is our first result.

\begin{theorem}\label{thm:main}
Let $0 < \alpha < d$, and let $\mu$ be a positive Borel measure
supported in $B(0,1)$, and $\gamma \in C^{d+1}([0,1])$ be a space
curve satisfying \eqref{eqn:non-degenerate}. Suppose
$I_\alpha(\mu)=1$, then for $\lambda
>1$ there exists a constant $C >0$ such that, for $\delta<\delta(\alpha)$,
\begin{equation}\label{eqn:curve}
\int_0^1 |\widehat \mu (\lambda \gamma(t))|^2 dt \le C  \lambda^{-\delta},
\end{equation}
where  $\delta(\alpha)=\frac{\alpha-d+2}{2}$ if  $d-1 \le \alpha < d$,
and $\delta(\alpha)=\max \Big(\frac{1-\langle d-\alpha
\rangle}{[d-\alpha]+1}, \frac{2-\langle d-\alpha
\rangle}{([d-\alpha]+1)(2-\langle d-\alpha \rangle) +1}\Big)$
otherwise.
\end{theorem}

For the case $d-1\leq\alpha<d$ the estimate is sharp except for the issue of the endpoint.  
But for the other case there is a gap between the
bound \eqref{eqn:curve} and the upper bounds which are obtained by
considering specific test examples. 
When $0 <\alpha\le 1$ we see from Theorem 1 in \cite{Er} that 
\eqref{eqn:curve} holds with $\delta \le\delta(\alpha)= \min (\alpha, 1/d)$ and this is optimal. 
(See Proposition \ref{pro:MSnece} for the upper bounds of $\delta$.)

In order to prove \eqref{eqn:curve},  instead of finiteness of
$\alpha$-dimensional energy $I_\alpha(\mu)$, it is convenient to
work with a growth condition on $\mu$.  We assume that there exists
a constant $C_\mu$, independent of $x$ and $r$, such that
\begin{align}\label{eqn:alpha-diml measure}
\mu(B(x,r)) \leq C_{\mu} r^{\alpha} \quad\textrm{for all } x \in \mathbb{R}^{d} \textrm{ and } r > 0.
\end{align}
It is clear  that \eqref{eqn:alpha-diml measure} implies that
$I_{\alpha-\epsilon}(\mu) < \infty$ for any $\epsilon >0$. The
converse is essentially true up to a logarithmic loss (for example,
see Lemma \ref{decomp}). For  $\mu$ satisfying \eqref{eqn:alpha-diml
measure} we set
\begin{align}\label{munorm}
    \langle\mu \rangle_\alpha = \sup_{(x,r)\in \mathbb R^d\times \mathbb R_+} {r^{-\alpha}} \mu(B(x,r)).
\end{align}

For the integral in the left hand side of \eqref{eqn:curve} it doesn't
seem easy to make use of the geometric feature of the curve $\gamma$. So we consider a
dual form which looks like Fourier restriction estimate. 
In fact, \eqref{eqn:curve} is equivalent to the estimate
 \[|\int
\widehat{g}d\mu| \le C\lambda^{\frac12(1 - \delta)}\|g\|_2\] when
$g$ is supported in $\lambda\gamma+O(1)$, the $O(1)$-neighborhood of the curve
$\lambda\gamma$. This can also be generalized by allowing
different orders of integrability.  We investigate
$\kappa=\kappa(q)$ for which
\begin{equation}\label{eqn:l2q}
    \|\widehat{g}\|_{L^{q}(d\mu)} \leq C \lambda^{\kappa} \|g\|_{L^2}
\end{equation}
holds for some $C>0$. This also has its own interest and for the
case of the circle the optimal results were obtained by  Erdo\u{g}an \cite{Er}.

\begin{figure}[t]\label{fig1}
\centerline{\epsfig{file = 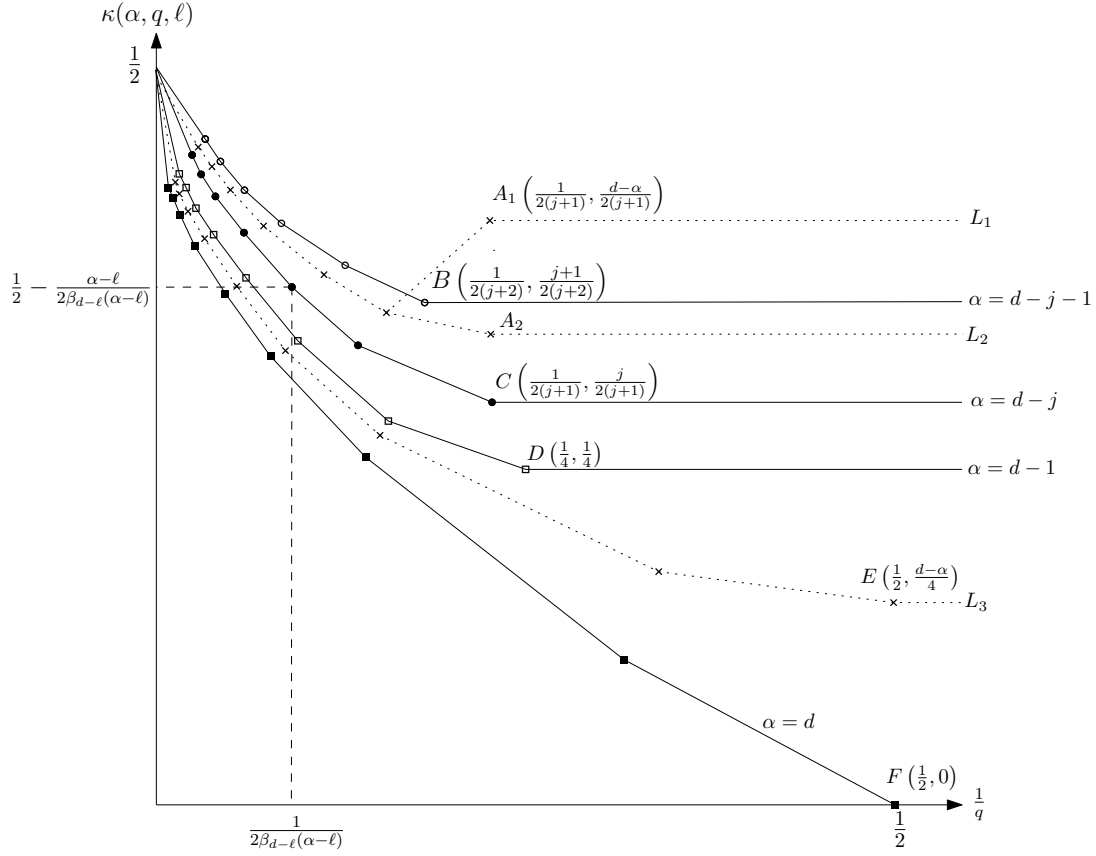, width=0.95\textwidth}}
\caption{\small 
The solid lines represent the value of $\kappa(\alpha,q,\ell)$ as a function of $1/q$ for specific values of $\alpha$, namely, $\alpha=d,\,d-1,\,d-j,\,d-j-1$ while $j$ is an integer, $1\le j<d-1$.
For integer $\alpha$, $\kappa(\alpha,q,\ell)$ decreases as so does $q$. 
The dotted graphs $L_1,L_2$ give the cases of non-integer $\alpha$ satisfying $d-j-1<\alpha<d-j$ and $1 \le j = [d-\alpha]$. 
If $\alpha <d-j-1 + ({j+1})/({j+2})$, $\kappa(\alpha,q,\ell)$ may increase. So, $\kappa(\alpha,q,\ell)$ may exceed $\kappa(d-j-1,q,\ell)$ at $A_1$. (See $L_1$.) However, 
if $\alpha$ is close enough to $d-j$,  the line of the shape like  $L_2$ appears. 
The dotted graph $L_3$ shows the case of non-integer for $\alpha \in (d-1,\,d)$. In this case, $\kappa(\alpha,q,\ell)$  always decreases in $q$. 
Except for $A_1,A_2,B,\dots,F$,  every marked dot  is given by  
$(\frac1q,\kappa(\alpha,q,\ell)) =$ $(\frac1{2\beta_{d-\ell}(\alpha-\ell)}, \frac12 - \frac{\alpha-\ell}{2\beta_{d-\ell}(\alpha-\ell)}).$
}
\end{figure}

 Now, to facilitate the statements of our results,  we define some
notations.
  For $j= 1,\ldots, d $ and $0<\alpha \le j$  we set
    \[ \beta_j(\alpha)=([j-\alpha]+1)\alpha+\frac{(j-1-[j-\alpha])(j-[j-\alpha] )}{2}.\]
For a fixed  $0<\alpha \le d$, we define the closed intervals $J(\ell)$,
$\ell = -1,0,1,\dots,d-1-[d-\alpha]$,   by setting
    \begin{align*}
    J(\ell ) =
    \begin{cases}
    \, [\, 2\beta_d(\alpha), \,\infty \, ], & \textrm{if } \, \ell=-1,\\
    \, [\, 2 \beta_{d-\ell-1}(\alpha-\ell-1), \,2 \beta_{d-\ell}(\alpha-\ell)\, ],
 &\textrm{if }\, 0 \le \ell \le d-3-[d-\alpha], \\
  \, [\, 2 ([d-\alpha]+1), \,2 \beta_{d-\ell}(\alpha - \ell )\,], &\textrm{if }\, \ell = d-2-[d-\alpha],\\
  \, [\,1,\, 2([d-\alpha]+1)\,], &\textrm{if }\, \ell = d-1-[d-\alpha].
 \end{cases}
    \end{align*}
Note that $\beta_{d-\ell}(\alpha-\ell)$ decreases as $\ell$ increases.
For each $\ell = -1,0,1,\dots,d-1-[d-\alpha]$ and $q\in  J(\ell)$, we also set 
\begin{align*}
\kappa(\alpha,q,\ell)=
\begin{cases}
\,   \frac 12 - \frac \alpha q,
             &\textrm{if }\,\,   \ell=-1, \\
\, \frac 12  - \frac{\alpha-\ell}q + \frac1{d-\ell} \big(  \frac{\beta_{d-\ell}(\alpha-\ell)}q - \frac 12 \big),
              &\textrm{if }\,\,     0\le \ell \le d-3-[d-\alpha], \\
\, 
\frac 12 \! - \! \frac{\alpha-\ell}q \! + \! \frac1{\mathfrak J_\ell} \big(  \frac{\beta_{d-\ell}(\alpha-\ell)}q \! - \! \frac 12 \big),		              
&\textrm{if }\,\,		\ell = d-2-[d-\alpha], \\
\, \min \big( \frac{d-\alpha}{4}, \frac{d-\alpha}{2([d-\alpha]+1)}\big)  , &\textrm{if }\,\, \ell=d-1-[d-\alpha], 
\end{cases}
\end{align*}
where $\mathfrak J_\ell = d-\ell = 2$ if $[d-\alpha]=0$, and $\mathfrak J_\ell=|J(d-2-[d-\alpha])|/2$ if $[d-\alpha] \ge 1$.
Here $|J(\ell)|$ denotes the length of $J(\ell)$. 
It should be
noted that, for given $\alpha$ and $\ell$, $\kappa(\alpha,q,\ell)$
is defined only for $q\in  J(\ell)$.
(See Figure \ref{fig1}.)

Our second result reads as the following from which Theorem
\ref{thm:main} is to be deduced later.

\begin{theorem}\label{thm:L^2-L^q est}
Let $0 < \alpha \le d$, and let $\gamma$ be given as in Theorem
\ref{thm:main}. Suppose that $\mu$ is supported in $B(0,1)$ and
satisfies \eqref{eqn:alpha-diml measure}. Then
\[  \|\widehat{g}\|_{L^{q}(d\mu)} \leq C\langle\mu\rangle_\alpha^{\,\,\frac1q} \lambda^{\kappa(\alpha,q,\ell)+\epsilon} \|g\|_{L^2}\]
holds for any $\epsilon
>0$ and for $q\in J(\ell)$, $\ell = -1,0,1,\dots,d-1-[d-\alpha]$. 
\end{theorem}

For a given $\alpha$, the results of Theorem \ref{thm:L^2-L^q est} are sharp for $q
\in J(\ell)$, $\ell \le d-3-[d-\alpha]$   in that the value $\kappa$ can not generally be made smaller except $\epsilon$. 
 For $q\in J(\ell)$, $\ell \ge d-2-[d-\alpha]$, the results are sharp only when $[d-\alpha]=0$.
In this case $\kappa (\alpha,q,d-2) = \frac 14 + \frac{d-\alpha-1}{2q}$ for $q \in J(d-2)$, which
is obtained by adapting the bilinear argument due to Erdo\u{g}an \cite{Er1}. (See Theorem \ref{thm:erd-ober}.)
It follows by H\"older's inequality that $\kappa(\alpha,q,d-1)=\frac{d-\alpha}{4}$ for $q \in J(d-1)$.
When $[d-\alpha]\ge1$ and $\alpha$ is an integer i.e. $\alpha = d-[d-\alpha]$, we have $\mathfrak J_\ell = |J(d-2-[d-\alpha])|/2 = d-\ell$. 
For this case, $\kappa(\alpha,q,d-2-[d-\alpha])$ are sharp.
In general, $\mathfrak J_\ell = |J(d-2-[d-\alpha])|/2 \le d-\ell$ for $[d-\alpha]\ge1$.
(See Proposition \ref{pro:nece}.)

\begin{rmk}\label{rmk}
If $\ell \le d-3-[d-\alpha]$,   $\kappa(\alpha,q,\ell)$ decreases as
so does $q$. However $\kappa(\alpha,q,d-2-[d-\alpha])$ may increase
though $q$ decreases except for the case $[d-\alpha] =0$.

As shown in Section \ref{sec:mainpf}, the decay rate $\delta$ in Theorem \ref{thm:main} is determined by the minimum of
$\kappa(\alpha,q,\ell)$, which is given by
$ \frac{d-\alpha}4$ if $[d-\alpha]=0$, or $\min_{q\in J(d-2-[d-\alpha])}  \kappa(\alpha,q, \ell)$ if $[d-\alpha]\ge 1$.
(See Figure 1.)
 
\end{rmk}

Although those notations seem to be complicated, most of them are
naturally associated with the scaling structure of curves.
For example, $\beta_j(\alpha)$ generalizes the number $ \beta_d(d)=d(d+1)/2 $ 
which appears in the studies on restriction estimates for space
curves (e.g. \cite{BL, BOS2, BOS1, BOS3,
DeMu, DeWr, DrMa, DrMa1}). 
We need to use the intervals $J(\ell)$ in order to extend the estimate \eqref{osc} beyond the known range given by \eqref{range} with $p=2$.
Except for the case $\ell = d-1-[d-\alpha]$  the number
$\kappa(\alpha,q,\ell)$ is actually obtained by interpolating the estimates for $q$  at the endpoints of $J(\ell)$.

The paper is organized as follows. In Section \ref{sec:osc}, we
prove various $L^p\rightarrow L^q$ estimates for the related oscillatory
integral operators (Theorem \ref{thm:osc-est}). In Section
\ref{sec:mainpf}, Theorem \ref{thm:L^2-L^q est} will be deduced from
the estimates in Section \ref{sec:osc} and  we prove Theorem
\ref{thm:main}. In Section \ref{sec:nece}, we discuss the upper
bounds of $\delta$ and the lower bounds of $\kappa$ which appear in
Theorems \ref{thm:main} and \ref{thm:L^2-L^q est}, respectively. In
Section \ref{sec:pf multi}, we provide proofs
of the estimates in Section 2 by making use of multilinear argument
in \cite{HaLe}. Also Theorem \ref{thm:erd-ober} will be proved in Section \ref{sec:pf bi} by adapting the bilinear argument due to Erdo\u{g}an \cite{Er}.

Throughout the paper the constant $C$ may vary from line to line and in
addition to $\,\widehat{}\,\,$ we also use $\mathcal F$ to denote
the Fourier transform.

\section{Oscillatory integral operators}\label{sec:osc}

For $\lambda\ge 1$  let us consider  an oscillatory integral operator defined by
   \begin{align*}
    \mathcal E^{\gamma}_{\lambda}f(x) = a(x) \int_{I} e^{i\lambda x \cdot \gamma(t)} f(t) dt ,
    \end{align*}
where $a$ is a bounded function supported in $B(0,1)$ with $\|a\|_\infty \le 1$.   
The estimate \eqref{eqn:l2q} can be deduced from the estimate
\begin{equation}\label{osc}
\|\mathcal E^{\gamma}_{\lambda} f\|_{L^{q}(d\mu)} \lesssim \lambda^{-\vartheta} \|f\|_{L^{2}(I)}.
\end{equation}
 In fact, $\lambda \gamma(t) +O(1)$ can be foliated into a set of $O(1)$-translations of the curve $\lambda\gamma$.
Then,  a simple change of variables,  Minkowski's inequality, and Plancherel's theorem together with  \eqref{osc}  give   \eqref{eqn:l2q} with $\kappa=\frac12-\vartheta$.
The converse also can be shown  by making use of the uncertainty principle.
See Lemma \ref{lem:equiv} for the details.

In the recent paper  \cite{HaLe}, two of the authors proved that if $\mu$ and $\gamma$ satisfy \eqref{eqn:alpha-diml measure} and \eqref{eqn:non-degenerate}, respectively,
then
\begin{equation}\label{osc1}
\|\mathcal E^{\gamma}_{\lambda} f\|_{L^{q}(d\mu)} \lesssim \lambda^{-\frac{\alpha}{q}} \|f\|_{L^{p}(I)}
\end{equation}
holds for $1 \le p,q \le \infty$ satisfying $d/q \leq 1 - 1/p$, $q \ge 2 d$ and
    \begin{equation}\label{range}
    \frac{\beta_d(\alpha)}q+\frac 1 p < 1, \quad q > \beta_d( \alpha) +1 .
    \end{equation}
We refer to \cite{HaLe} and references therein for further discussions about this estimate and related results.
Then from Lemma \ref{lem:equiv} it follows that \eqref{eqn:l2q} holds with $\kappa=\frac12-\frac\alpha q$ if  $q > \max(2\beta_d(\alpha), 2d)$ and $\lambda >1$.
However this is not enough in order to obtain  the estimate \eqref{eqn:l2q} for the other $q$.
Hence we are led to investigate the estimates with $(p,q)$ which does not satisfy \eqref{range}.
It is natural to expect that the decay gets worse as $(1/p, 1/q)$ gets away from the range \eqref{range}.
If $\alpha=d$, then by  the  Lebesgue-Radon-Nikodym theorem we have $d\mu=f(x)
dx$ and by the  Lebesgue differentiation theorem and \eqref{eqn:alpha-diml measure} it follows that $f$ is
a bounded function. Hence, by projection argument, it is not
difficult to see that, for $k=0, \dots, d$,
    \[
    \|\mathcal E_\lambda^\gamma f\|_{L^q(d\mu)}\le C\lambda^{-\frac{k}q}\|f\|_p
    \]
whenever $\frac{\beta_k(k)}q+\frac1p \le 1 $. But this argument
readily fails with a general measure $\mu$. To get around this difficulty we make use of
the induction argument based on multilinear estimates (see \cite{HaLe} and \cite{BoGu}).

The following is an extension of the earlier result in \cite{HaLe}.

\begin{theorem}\label{thm:osc-est}
Let $\gamma$ and $\mu$ be given as in Theorem \ref{thm:L^2-L^q est}.
For each integer $\ell = 0, 1,\dots, d-1-[d-\alpha]$, there exists a constant $C_\ell$ such that
    \begin{equation}\label{eqn:osc1}
    \|\mathcal E_\lambda^{\gamma} f\|_{L^q(d\mu)}\le
    C_\ell \,\langle\mu\rangle_\alpha^{\,\,\frac1q}\, \lambda^{-\frac{\alpha-\ell}q}\|f\|_{L^{p}(I)}
    \end{equation}
holds for $f \in L^{p}(I)$ and $\lambda \ge 1$ whenever $(d-\ell)/q + 1/p \leq 1 $, $q \ge 2(d-\ell)$ and
    \begin{equation}\label{eqn:range}
    \frac{\beta_{d-\ell}(\alpha-\ell)}q+\frac1p < 1, \quad q > \beta_{d-\ell}(\alpha-\ell) +1 .
    \end{equation}
\end{theorem}

Theorem \ref{thm:osc-est}  is  proved by routine adaptation of  the argument in \cite{HaLe}. Compared to \cite{HaLe}
the main difference here is to utilize various multilinear estimates of different degrees of multilinearity.
For completeness we provide a proof of Theorem \ref{thm:osc-est}  in Section \ref{sec:pf multi}.

\begin{rmk}\label{rmk:range}
It is easy to check that among the four conditions on $(p,q)$ above, the first two conditions become redundant for some  $\ell$. In fact, since $\beta_{d-\ell}(\alpha-\ell)>d-\ell$ if and only if $\alpha-\ell>1$, and $\beta_{d-\ell}(\alpha-\ell)+1>2(d-\ell)$ if and only if $\alpha-\ell>2$, the estimate \eqref{eqn:osc1} holds whenever
\begin{align*}
    \begin{cases}
        \frac{\beta_{d-\ell}(\alpha-\ell)}q+\frac1p < 1,\,  q > \beta_{d-\ell}(\alpha-\ell) +1,  & \text{if }\, 2<\alpha-\ell\hspace{8.05 mm}  (\text{i.e. }\ell \leq d-3-[d-\alpha]) , \\
        \frac{\beta_{d-\ell}(\alpha-\ell)}q+\frac1p < 1,\, q \ge 2(d-\ell), & \text{if }\, 1<\alpha-\ell\le 2\,\, (\text{i.e. } \ell=d-2-[d-\alpha]) , \\
        \frac{ d-\ell}{q} + \frac 1 p \le 1 ,\,  q \ge 2(d-\ell), & \text{if }\, 0 < \alpha-\ell\le 1 \,\, (\text{i.e. } \ell = d-1-[d-\alpha]) .
    \end{cases}
\end{align*}
\end{rmk}

If  $[d-\alpha] \ge 1$, we also have estimates for $p,$ $q$ satisfying  $([d-\alpha] +1)/q +1/p >1$ and $q <2 ([d-\alpha] + 1)$,  
which are given as follows. 

\begin{theorem}\label{thm:osc-est2}
Suppose that $\gamma,\mu$ are given as in Theorem \ref{thm:osc-est}. Then, there exists a positive constant $C$ such that
\begin{align*}
  \|\mathcal E^{\gamma}_{\lambda}f\|_{L^{q}(d\mu)} \le C \langle\mu\rangle_\alpha^{\,\,\frac1q}\lambda^{-\big( \frac{\alpha - d}{q} + 1- \frac 1p \big)} \|f\|_{L^{p}(I)},
\end{align*}
whenever $ ([d-\alpha]+1)^{-1}(1 -\frac 1p )\le \frac{ 1}q \le  1- \frac 1p $, $q \ge [d-\alpha]+1$ and
$\frac{\beta_{[d-\alpha] +1}(1-\langle d-\alpha \rangle)}{q} + \frac 1p < 1$.
\end{theorem}

\noindent There is no reason to believe these estimates are sharp. Particularly, if  $[d-\alpha] =0$
(this gives the condition that $1/p + 1/q \le 1$ and $q \ge 2$), Theorem \ref{thm:osc-est2} coincides with Theorem \ref{thm:osc-est} for $\ell = d-1$.
See Section \ref{sec:pf multi} for a proof,  which is based on the generalized Hausdorff-Young inequality.

By interpolating the estimates \eqref{eqn:osc1}  for which  $(1/p,1/q)$  is near the critical line one can improve the bound.
To state this, we define some notations.
In addition, let us assume that $p \le 2$ for simplicity.
For each $\alpha$ let $\mathscr A(\ell) $ be the set of $(\frac1p,\frac1q)$ such that $1\le p \le 2$ and
\begin{align*}
\begin{cases}
\, \frac{\beta_{d}(\alpha)}{q} +\frac 1 p < 1 , &\textrm{ if } \ell = -1,\\
\, \frac{\beta_{d-\ell-1}(\alpha-\ell-1)}{q} +\frac 1 p < 1 \le \frac{\beta_{d-\ell}(\alpha-\ell)}{q} +\frac 1 p,
&\textrm{ if }  \ell=0, \dots, d-3-[d-\alpha], \\
\, \frac{[d-\alpha]+1}{q} +\frac 1 p \le 1 \le \frac{\beta_{[d-\alpha]+2}(2-\langle d-\alpha\rangle)}{q} +\frac 1 p,
&\textrm{  if } \ell = d-2-[d-\alpha].
\end{cases}
\end{align*}
Let us also denote by $\mathscr A(d-1-[d-\alpha])$ the set of $(p,q)$ satisfying the condition given in Theorem \ref{thm:osc-est2} and $1\le p \le 2$. Note that $\mathscr A(d-1)$ when $[d-\alpha ]=0$ represents the line segment $1/q + 1/p=1$ and $1\le p \le2$.

By interpolating the estimates in Theorem \ref{thm:osc-est} and Theorem \ref{thm:osc-est2}, we obtain the following.
\begin{corollary}\label{cor:osc-interp} Let $\gamma$ and $\mu$ be defined as in Theorem \ref{thm:osc-est}.
Suppose \eqref{eqn:osc1} holds. 
Then, for $1\le p\le 2$, there exists a constant $C>0$ such that, for any $\epsilon > 0$,
    \begin{align}\label{eqn:osc-interp}
    \|\mathcal E^{\gamma}_{\lambda}f\|_{L^{q}(d\mu)} \leq C \langle\mu\rangle_{\alpha}^{\,\,\frac 1q}\lambda^{-\eta(\alpha,p,q,\ell)+\epsilon}    \|f\|_{L^{p}(I)},
    \end{align}
where \begin{align*}
\eta(\alpha,p,q,\ell) = \begin{cases}
\,   \frac \alpha q,
&\textrm{if }\,   (\frac1p,\frac1q) \in \mathscr A(-1), \\
\,  \frac{\alpha-\ell}q - \frac2{|J(\ell)|} \left(  \frac{\beta_{d-\ell}(\alpha-\ell)}q + \frac 1p-1 \right),
 &\textrm{if }\,    (\frac1p,\frac1q) \in \mathscr  A(\ell) ,\, 0 \le \ell \le d-2-[d-\alpha], \\
 \, \frac{\alpha - d}{q} + 1- \frac 1p, &\textrm{if }\, (\frac1p,\frac1q) \in \mathscr A(d-1-[d-\alpha]).
\end{cases}
\end{align*}
\end{corollary}
Note that if $0 < \alpha \le 1$, we have only $\ell = 0$.
In this case, there is nothing to interpolate.
The results in Corollary \ref{cor:osc-interp} are sharp for $-1 \le \ell \le d-3-[d-\alpha]$ except $\epsilon$-loss. This can be shown
by the same examples which are used for the proof of Proposition \ref{pro:nece}.


\section{Proof of Theorem \ref{thm:main} and \ref{thm:L^2-L^q est}}\label{sec:mainpf}

As mentioned in the previous section, we will apply the decay
estimate for the related oscillatory integral operator to obtain
\eqref{eqn:l2q}. In this section we may assume that $\gamma$ is
close to $\gamma_\circ^d$ so that
 $\|\gamma-\gamma_\circ^d\|_{C^{d+1}(I)} \le \epsilon$ for any given $\epsilon>0$.  Here
$\gamma_\circ^d$ is defined by \eqref{gamma-k}. This
 can be justified easily  by   decomposing the curve $\gamma$ into a finite union of
(sub)curves, rescaling and using Lemma \ref{lem:normalization}.

We start with observing that  \eqref{osc} is equivalent to the
estimate \eqref{eqn:l2q}.

\begin{lemma}\label{lem:equiv}
    Let $d \geq 2$ and $0 < \alpha \leq d$.
    Suppose that $\gamma \in C^{d+1}(I)$ satisfies \eqref{eqn:non-degenerate} and $\mu$ is a positive Borel measure supported in
    $B(0,1)$ satisfying \eqref{eqn:alpha-diml measure}.
    The estimate \eqref{osc} holds with $\vartheta=\frac12 -\kappa$  if and only if  the estimate \eqref{eqn:l2q} holds whenever  $g$ is supported in $\lambda\gamma(I)+O(1)$.
\end{lemma}

\begin{proof}
First we show that \eqref{osc} implies \eqref{eqn:l2q}. Let $g$ be a function which is supported in $\lambda\gamma(I)+O(1)$.
By the change of variables $\xi\to \lambda\xi$, we may write
\begin{align*}
\widehat{g}(x) = \int_{\gamma(I)+O(\lambda^{-1})} e^{i\lambda x \cdot \xi} \lambda^{d} g(\lambda \xi) d\xi .
\end{align*}
Let us consider a nondegenerate curve $\gamma_\ast$ which is given by extending $\gamma$ to the interval
$I_\ast := [-\frac{C}{\lambda}, 1 + \frac{C}{\lambda}]$
such that $\gamma_{\ast} = \gamma$ on $I$ and $\| \gamma_* -\gamma_\circ^d\|_{C^{d+1}( I_* )} \le \epsilon$ for a sufficiently small $\epsilon>0$.
Then it follows that, for a sufficiently large constant $C$,
\begin{align*}
\gamma(I) + O(\lambda^{-1})
\subset
\{\gamma_{\ast}(s) + (0,\mathbf{v}) : s \in {I_\ast}, \mathbf{v} \in \mathbb{R}^{d-1} \textrm{ satisfying } |\mathbf{v}| \leq C \lambda^{-1}\} .
\end{align*}

Let us define a map $\Gamma : I_\ast \times\mathbb{R}^{d-1} \to \mathbb{R}^{d}$ by
$\Gamma(s,\mathbf{v}) = \gamma_{\ast}(s) + (0,\mathbf{v}) $.
Then $|\det\frac{\partial\Gamma}{\partial(s,\mathbf{v})}|\ge c> 0$.
Thus we have 
\begin{align*}
\widehat{g} (x)
= C \int_{|\mathbf v|\lesssim \lambda^{-1}} \int_{ I_*}
e^{i \lambda x \cdot (\gamma_{\ast}(s) + (0,\mathbf{v}))} \lambda^{d} \widetilde  g(\lambda(\gamma_{\ast}(s) + (0,\mathbf{v}))) ds d\mathbf{v}
\end{align*}
with $|\widetilde g|\lesssim  |g| $. By setting $\widetilde{\gamma}(t) = \gamma_{\ast} ( ( 1 + 2C/{\lambda} )t - C/{\lambda} )$,
we have a nondegenerate curve  $\widetilde \gamma $ defined  on $I$ which is still close to $\gamma_\circ^d$. Then, it follows that
\begin{align*}
|\widehat{g}(x)|
\leq C \int_{|\mathbf v|\lesssim \lambda^{-1}}
\Big| \int_{I} e^{i \lambda x \cdot \widetilde{\gamma}(t)} \lambda^{d} \widetilde g(\lambda(\widetilde{\gamma}(t) + (0,\mathbf{v}))) dt
\Big| d\mathbf{v} .
\end{align*}

After Minkowski's inequality, we apply \eqref{osc} by freezing $\mathbf v$
 to see that
\begin{align*}
\| \widehat{g}\|_{L^{q}(d\mu)}
\leq C \int_{|\mathbf v|\lesssim \lambda^{-1}} \lambda^{-\vartheta} \|f_{\mathbf{v}}\|_{L^{2}(I)} d\mathbf{v} ,
\end{align*}
where $f_{\mathbf{v}}(t) := \lambda^{d} \widetilde g(\lambda(\widetilde{\gamma}(t)+(0,\mathbf{v})))$.
By the Cauchy-Schwarz inequality, we get
\begin{align*}
\| \widehat{g} \|_{L^{q}(d\mu)}
\leq C \lambda^{-\vartheta-\frac{d-1}{2}}
\Big( \int_{|\mathbf v|\lesssim \lambda^{-1}} \int_{I} |f_{\mathbf{v}}(t)|^{2} dt d\mathbf{v} \Big)^{\frac{1}{2}}
\leq C \lambda^{\frac{1}{2}-\vartheta}
\Big(\int_{\lambda \gamma(I) + O(1)} |g( \xi)|^{2} d\xi \Big)^{\frac{1}{2}} ,
\end{align*}
which implies \eqref{eqn:l2q}.

Conversely, let us show that \eqref{eqn:l2q} implies \eqref{osc}.
For $\mathbf{v} = (v_{2},\cdots,v_{d}) \in \mathbb{R}^{d-1}$ as above, one easily sees that
\begin{align}\label{eqn:d-1}
\Big| \int_{I} e^{i\lambda x \cdot \gamma(t)} a(x) f(t) dt \Big|
= \lambda^{d-1}
\Big| \int_{|\mathbf v|\lesssim \lambda^{-1}} \int_{I} e^{i \lambda x \cdot (\gamma(t) + (0,\mathbf{v}))} a(x) f(t) dt \, e^{-i \lambda x \cdot (0,\mathbf{v})} d\mathbf{v} \Big|.
\end{align}
By expanding into power series we  write
$
e^{-i \lambda x\cdot (0,\mathbf{v})}
= \sum_{\eta, \eta'} c_{\eta, \eta'} x^{\eta} (0,\lambda \mathbf{v})^{\eta'},
$
where $\eta, \eta'$ denote multi-indices. Then it is easy to see
$ \sum_{\eta, \eta'} |c_{\eta,\eta'}| R^{|\eta|+|\eta'|}\lesssim e^{dR^2} $.
Since $\mu$ is supported in $B(0,1)$,
setting $G_{\eta'}(t,\mathbf{v}) := f(t) (0,\lambda \mathbf{v})^{\eta'} $ gives
\begin{align*}
\eqref{eqn:d-1}
&
 \lesssim \sum_{\eta,\eta'} |c_{\eta,\eta'}| \lambda^{d-1} \Big| \int_{I} \int_{|\mathbf v|\lesssim \lambda^{-1}}
e^{i \lambda x \cdot (\gamma(t) + (0,\mathbf{v}))} G_{\eta'}(t,\mathbf{v})
d\mathbf{v} dt \Big|.
\end{align*}
By the change of variables $\xi= \gamma(t) + (0,\mathbf{v})$, we
obtain
\begin{align*}
 \Big| \int_{I} \int_{|\mathbf v|\lesssim \lambda^{-1}}
e^{i \lambda x \cdot (\gamma(t) + (0,\mathbf{v}))} G_{\eta'}(t,\mathbf{v})
d\mathbf{v} dt \Big|
= \lambda^{-d}\Big| \int_{\lambda\gamma(I) + O(1)} e^{i x \cdot \xi} g_{\eta'}(\lambda^{-1}\xi) d\xi \Big|
\end{align*}
where $g_{\eta'}(\xi) = G_{\eta'}(t(\xi),\mathbf v(\xi))$.
Hence, using \eqref{eqn:l2q} and Minkowski's inequality and reversing the change of variables we see
\begin{align*}
\Big\| \int_{I} e^{i\lambda x \cdot \gamma(t)} &a(x) f(t) dt \Big\|_{L^q(d\mu)}
\lesssim \lambda^{-1 + \kappa} \sum_{\eta,\eta'} |c_{\eta,\eta'}| \, \| g_{\eta'}(\lambda^{-1} \cdot ) \|_{L^2(\mathbb R^d)} \\
& \lesssim \lambda^{-1+\frac d2 + \kappa} \sum_{\eta,\eta'} |c_{\eta,\eta'}|
\Big(\int_{|\mathbf v|\lesssim \lambda^{-1}}\int_{I} | f(t) (0,\lambda \mathbf{v})^{\eta'} |^{2} dt d\mathbf{v}  \Big)^\frac12\\
& \lesssim \lambda^{-\frac12 +\kappa} \sum_{\eta,\eta'} |c_{\eta,\eta'}| C^{|\eta'|} {\|f \|_{L^2(I)}}
\lesssim \lambda^{\kappa - \frac 12} \| f\|_{L^2(I)}.
\end{align*}
The third inequality follows from $|(0,\lambda\mathbf{v})| \lesssim 1$.
This completes the proof. 
\end{proof}

Now we prove Theorem  \ref{thm:L^2-L^q est}.

\begin{proof}[Proof of Theorem  \ref{thm:L^2-L^q est}]
By Lemma \ref{lem:equiv} and Corollary \ref{cor:osc-interp} with $p=2$, it follows that
the estimate \eqref{eqn:l2q} holds with 
\begin{align}\label{etta}
 \kappa = \frac12 - \eta(\alpha,2,q,\ell) + \epsilon  
\end{align}
for $\epsilon>0$ and  $-1 \le \ell \le  d-2-[d-\alpha]$.
For these $\ell$,  $(\frac12,\frac1q)\in \mathscr A(\ell)$ if and only if $q\in  J(\ell)$.

Now we consider  the case $q \in J(d-1-[d-\alpha])$ ($\ell=d-1-[d-\alpha]$).
By the same argument as in the above, using Lemma \ref{lem:equiv} and Corollary \ref{cor:osc-interp} with $p=2$, we get \eqref{eqn:l2q}
with $\kappa=(d-\alpha)/q+\epsilon \ge (d-\alpha)/2([d-\alpha]+1) +\epsilon$ for
$\epsilon>0$ and $[d-\alpha]+1 \le q \le 2([d-\alpha]+1)$ (which coincides with $\mathscr A(d-1-[d-\alpha])$ for $p=2$).
Since $\mu$ is a finite measure, the range of $q$ can be extended by H\"{o}lder inequality.
Thus we obtain \eqref{eqn:l2q} with 
\begin{align}\label{eend}
 \kappa = \frac{d-\alpha}{2([d-\alpha]+1)} +\epsilon 
\end{align}
for $q \in J(d-1-[d-\alpha]) = [1,2([d-\alpha]+1)]$.
Hence \eqref{etta} and \eqref{eend} correspond to $\kappa(\alpha,q,\ell) +\epsilon$ for $[d-\alpha]\ge1$ because $|J(\ell)| = 2(d-\ell)$ for $\ell \le d-3-[d-\alpha]$.

As mentioned above, for $[d-\alpha]=0$ (and $\ell=d-2-[d-\alpha], d-1-[d-\alpha]$),   a better estimate is possible by making use of the bilinear approach
(see Erdo\u{g}an \cite{Er}). The following is proved in Section
\ref{sec:pf bi}.

\begin{theorem}\label{thm:erd-ober}
    Suppose that $d \geq 2$ and $d-1\leq \alpha \leq d$.
    Let $\gamma$, $\mu$, and $g$ be given as in Theorem \ref{thm:L^2-L^q est}.
        Then, for $\lambda > 1$, $q \geq 2$ and $\epsilon > 0$,
    there exists a constant $C >0$ such that  
    \begin{equation*}
     \|\widehat{g}\|_{L^{q}(d\mu)} \leq C\langle\mu\rangle_\alpha^{\,\,\frac1q} \lambda^{\kappa+\epsilon} \|g\|_{L^2}    \end{equation*}
    for $ \kappa=\max(\frac14 + \frac{ d -\alpha-1}{2q},\,\frac12 + \frac{d-\alpha-2}{q}) $.
 \end{theorem}

This gives
\[
\kappa >
\begin{cases}
\, \frac 14 + \frac {d-\alpha-1}{2q}, &\textrm{ if } 2 \le q \le 2(\alpha -d +3),\\
\, \frac 12 + \frac{d-\alpha-2}{q}, &\textrm{ if } 2(\alpha-d +3) \le q.
\end{cases}
\]
Since $[2, \, 2(\alpha-d+3) ]= J(d-2)$ if $[d-\alpha]=0$, \eqref{eqn:l2q} holds with   $\kappa=\frac 14 + \frac {d-\alpha-1}{2q}+\epsilon = \kappa(\alpha,q,d-2) +\epsilon$ for $q\in J(d-2)$, $\epsilon>0$.
By taking $q=2$ and using H\"{o}lder's inequality, we get
 $\kappa = \frac{d-\alpha}{4} +\epsilon  = \kappa(\alpha,q,d-1) +\epsilon $ for $q \in J(d-1) = [1,2]$.
This completes the proof. 
\end{proof}

Now we turn to the proof of Theorem \ref{thm:main} for which we need the following lemma.

\begin{lem}\label{l222} Let $\mu$ be  a finite measure which is supported in  $B(0,1)$. Suppose that the estimate
\begin{equation}\label{l22}
\big|\int \widehat{g}(x) d\mu(x)\big| \lesssim {\sqrt{I_\alpha(\mu)}} \lambda^{\kappa} \|g\|_2
\end{equation}
    holds whenever $g$ is supported in $\lambda\gamma(I) + O(1)$. Then \eqref{eqn:curve} holds with $\delta=1-2\kappa$.
\end{lem}

\begin{proof}
The proof is a simple modification of the argument in \cite{Wo1} (see also \cite{ErOb}).
By the assumption \eqref{l22} and duality, we have 
\begin{align}\label{eqn:decay uncert}
        \int_{\lambda\gamma(I) + O(1)} |\widehat{\mu}(\xi)|^{2} d\xi \lesssim I_\alpha(\mu) \lambda^{2\kappa}.
    \end{align}

    Let  $\psi$ be a Schwartz function which is equal to $1$ on the support of $\mu$. Then
    \begin{align*}
        \int_{I} |\widehat{\mu}(\lambda\gamma(t))|^{2} dt
        = \int_{I} |\widehat{\psi} \ast \widehat{\mu}(\lambda\gamma(t))|^{2} dt
        \lesssim \int_{\mathbb R^d} \int_{I} |\widehat{\psi}(\lambda\gamma(t)-\xi)| dt |\widehat{\mu}(\xi)|^{2} d\xi .
    \end{align*}
    By rapid decay of $\widehat\psi$,
    $\int_{I} |\widehat{\psi}(\lambda\gamma(t)-\xi)| dt
    \lesssim \lambda^{-1}{(1+\dist(\lambda\gamma(I),\xi))^{-N}}$
    for a sufficiently large $N \ge d$.
    Hence, it follows that
    \begin{align}\label{eqn:change-lambda}
        \int_{I} |\widehat{\mu}(\lambda\gamma(t))|^{2} dt \lesssim \frac{1}{\lambda} \int \frac{|\widehat{\mu}(\xi)|^{2}}{(1+\dist(\lambda\gamma(I),\xi))^{N}} d\xi .
    \end{align}
    By dyadic decomposition along the distance between $\xi$ and $\lambda\gamma(I)$, we see
    \begin{align*}
        \int \frac{|\widehat{\mu}(\xi)|^{2}}{(1+\dist(\lambda\gamma(I),\xi))^{N}} d\xi
        & \lesssim \int_{\lambda\gamma(I)+O(1)} |\widehat{\mu}(\xi)|^{2} d\xi +  \sum^{\infty}_{  j=1} 2^{-N j} \int_{\lambda\gamma(I)+O(2^{j})} |\widehat{\mu}(\xi)|^{2} d\xi \\
        & \lesssim I_\alpha(\mu)\lambda^{2\kappa} + I_\alpha(\mu) \sum^{\infty}_{j=1} 2^{-Nj} 2^{(d-1)j} \lambda^{2\kappa} \lesssim I_\alpha(\mu)\lambda^{2\kappa} .
    \end{align*}
    The second inequality follows from the fact that $\lambda \gamma(I) + O(2^j)$ is a union of  
    translations of $\lambda \gamma(I) +O(1)$.
    Consequently, by combining this and \eqref{eqn:change-lambda} we obtain \eqref{eqn:curve} with $\delta=1-2\kappa$.
\end{proof}

We also need the following lemma due to Wolff \cite[Lemma 1.5]{Wo1}.
In \cite{Wo1} the proof of this lemma is given only for $d=2$ but
the argument works for any dimension.

\begin{lem}\label{decomp}  Let $\mu$ be a positive Borel  measure supported in $B(0,1)$. Then, for $R>1$,  $\mu$ can be written as
$ \mu=\sum_{1\le j\le O(\log R)} \mu_j$ such that $\mu_j$ is a positive Borel measure supported in $B(0,1)$ and, for each $j$,
\begin{equation}\label{RR}
\mu_j(\mathbb R^d) \sup_{(x,r)\in \mathbb R^d\times [R^{-1}, \infty) } r^{-\alpha} \mu_j(B(x,r))\lesssim {I_\alpha(\mu)}.
\end{equation}
\end{lem}

\begin{proof}[Proof of Theorem \ref{thm:main}]By Lemma \ref{l222}, for
\eqref{eqn:curve} we need to show  \eqref{l22}. Now, by Lemma
\ref{decomp} with $R=\lambda$ there are as many as $O(\log \lambda)$
measures. Ignoring logarithmic loss we may consider only one  of such measures $\mu$
 which satisfies \eqref{RR}, and we need to show that,
for {$\kappa > (1-\delta(\alpha))/2$},
\begin{equation}\label{l2222}
\big|\int \widehat{g}(x) d\mu(x)\big| \leq C\mu(\mathbb
R^d)^\frac12\langle\mu\rangle_\alpha^{\,\,\frac12} \lambda^{\kappa}
\|g\|_2
\end{equation}
holds whenever $g$ is supported in $\lambda\gamma(I) + O(1)$ and
$\mu$ is a positive Borel measure supported in $B(0,1)$ satisfying
\eqref{RR}. However, we may assume a stronger condition $\mu(\mathbb
R^d)\langle\mu\rangle_\alpha\le {I_\alpha(\mu)}$ holds. In fact, since $g$ is
supported in $\lambda\gamma (I)+O(1)$, the estimate we need to show is
equivalent to
\[\big|\int \mathcal F({\psi(\cdot/ \lambda)g})(x) d\mu(x)\big| \leq C\mu( {\mathbb
R^d})^\frac12\langle\mu\rangle_\alpha^{\,\,\frac12} \lambda^{\kappa}
\|g\|_2,\] where $\psi$ is a Schwartz function with $\psi\sim 1$ on
the ball $B(0,Cd)$ and with $\widehat \psi$ supported in $B(0,1)$.
Since $\mathcal F({\psi(\cdot/ \lambda)g})=\lambda^d\widehat{\psi}(\lambda\,\cdot)\ast
\widehat g$, we may replace $d\mu$ with $\lambda^d d\mu\ast
|\widehat{\psi}|(\lambda\,\cdot)$. Then it is easy to see that  $\lambda^d
d\mu\ast |\widehat{\psi}|(\lambda\,\cdot)(\mathbb R^d)\lesssim \mu(\mathbb
R^d)$ and 
$\lambda^d d\mu\ast |\widehat{\psi}|(\lambda\,\cdot) (B(x,r))\lesssim
\sup_{(x,r)\in \mathbb R^d\times [\lambda^{-1}, \infty) }  r^{-\alpha}
\mu (B(x,r)) $ for  $r>0$.

Since $\mu$ is supported in $B(0,1)$ with $\mu(\mathbb{R}^d)\langle\mu\rangle_\alpha \leq I_\alpha(\mu)$ and $q\ge 2$, by H\"older's inequality and Theorem \ref{thm:L^2-L^q est}
we get, for $\kappa> \kappa(\alpha,q,\ell)$,
\[ \int |\widehat{g}(x)| d\mu(x) \le \|\widehat g\|_{L^q(d\mu)}\mu(\mathbb R^d)^{1-\frac1q}
\lesssim \mu(\mathbb
R^d)^{1-\frac1q}\langle\mu\rangle_\alpha^{\,\,\frac1q}
\lambda^{\kappa} \|g\|_{2}.\] 
Clearly, $\mu(\mathbb R^d)\lesssim
\langle\mu\rangle_\alpha$ because $\mu$ is supported  in $B(0,1)$. Hence
we have  \eqref{l2222} whenever $\kappa> \kappa(\alpha,q,\ell)$ with $q\ge 2$. 
Therefore we only have to check the minimum of
$\kappa(\alpha,q,\ell)$, $q\in J(\ell)$ which depends on $\alpha$.

  First we consider the case  $d-1\leq \alpha < d$.
    It is easy to see that $\min_{\ell}  \min_{q\in J(\ell)\cap[2,\infty)}  \kappa(\alpha,q,\ell)$ $= \kappa(\alpha,2,d-2)= \frac{d-\alpha}{4}.$
    Thus we obtain the first part of Theorem \ref{thm:main}.

    For the case $[d-\alpha] \ge 1$, finding the minimum of $ \min_{q\in J(\ell)\cap[2,\infty)} \kappa(\alpha,q,\ell)$ is less obvious.
As mentioned in Remark \ref{rmk}, the minimum occurs when $q \in J(d-2-[d-\alpha])$. 
In fact, $\min_{q \in J(d-2-[d-\alpha])} \kappa(\alpha,q,\ell)$ is given by 
\[
     \frac 12 - \frac{\alpha - d+ 2+[ d-\alpha] }{2 \beta_{[d-\alpha]+ 2}(\alpha-d+ 2 +[d-\alpha])}
     = \frac12 - \frac{2-\langle d-\alpha\rangle}{2 ([d-\alpha]+1)(2-\langle d-\alpha\rangle) +2}
     \]
with $q = 2 \beta_{[d-\alpha]+2}(2-\langle d-\alpha\rangle)$, 
or $\frac {d-\alpha}{2([d-\alpha]+1)}$ with $q = 2([d-\alpha]+1)$.
    Combining these two gives the other part of Theorem \ref{thm:main}.
    This completes the proof. 
\end{proof}

\section{Upper bound for $\delta$ and lower bound for $\kappa$}\label{sec:nece}

In this section we consider the upper bound for $\delta$ and the
lower bound for $\kappa$ which limit the values $\delta$, and
$\kappa$ in the estimates  \eqref{eqn:curve} and \eqref{eqn:l2q}. As
mentioned before, for the former there is a gap between our result
{and} the plausible upper bound stated in Proposition
\ref{pro:MSnece}. For the latter, the bounds we obtain here turn out
to be sharp in various cases.

\begin{proposition}\label{pro:MSnece}
    Let $0 < \alpha <d$ and $\gamma$ be given as in Theorem
    \ref{thm:main}.
    Suppose \eqref{eqn:curve} holds uniformly whenever
    $I_\alpha(\mu)=1$. Then, for $[d-\alpha] = 1, \cdots, d-2$,
    \begin{subnumcases}{\delta \leq}
            1 - \tfrac{d-\alpha}{2}, & if $\alpha\in(d-1,\, d)$,\label{sub:a}   \\
            \min\big(  1 - \tfrac{d-\alpha}{[d-\alpha]+2}, \tfrac{1}{[d-\alpha]+1}\big), & if $\alpha\in\big(d-[d-\alpha]-1, \,d-[d-\alpha] \big]$,\label{sub:b}  \\
            \min\big(\alpha, \tfrac{1}{d} \big), & if $\alpha\in(0,\,1]. \label{sub:c}$
    \end{subnumcases}
\end{proposition}

Thus we see that \eqref{eqn:curve} is sharp when $d-1 \leq \alpha <
d$. 
As mentioned before, Theorem 1 in \cite{Er} shows that \eqref{sub:c} is sufficient for \eqref{eqn:curve} to hold when $\alpha \in (0,1]$.

\begin{proof}[Proof of Proposition \ref{pro:MSnece}]
    For a given $[d-\alpha]$, let us fix an integer $\ell$ such that $0\le \ell\le d-[d-\alpha]-1$.
Let $\psi$ be a Schwartz function supported in $B(0,2)$ with $\|\psi\|_{L^1}=1$. 
We also set
\[
\psi_{\lambda,\ell} (x)
= \lambda^{-\frac{1-d+\ell}{2}} \psi(\lambda^{1-\frac{1}{d-\ell}} x_1, \dots, \lambda^{1-\frac{d-\ell}{d-\ell}} x_{d-\ell},x_{d-\ell+1},\dots,x_d),
\]
so that $\|\psi_{\lambda,\ell} \|_{L^1} =1$. Then there exists a rectangle  $S_\ell$  such that $ |\widehat {\psi_{\lambda,\ell}}| \sim 1$
on $S_\ell$, where $S_\ell$ is a $d$-dimensional rectangle defined by
    \begin{align*}
        S_\ell=\Big\{x\in \mathbb{R}^{d} : |x_{1}|\lesssim\lambda^{1-\frac{1}{d-\ell}}, |x_{2}|\lesssim\lambda^{1-\frac{2}{d-\ell}},
        \cdots, |x_{d-\ell}|\lesssim\lambda^{1-\frac{d-\ell}{d-\ell}}=1,\cdots, |x_{d}|\lesssim 1\Big\}.
    \end{align*}
By Taylor's expansion,
    we have
    \begin{align}\label{eqn:taylor at 0}
        \gamma(t) - \gamma(0)
        = \gamma'(0)t + \gamma''(0)\frac{t^{2}}{2!} + \cdots + \gamma^{(d)}(0)\frac{t^{d}}{d!} + \mathbf e(t)
        =: M^{\gamma,d}_{0}\gamma_{\circ}^d(t) + \mathbf e(t),
    \end{align}
where $M_0^{\gamma,d}$ is a nonsingular matrix given by
\eqref{eqn:M} and $|\mathbf e(t)| \lesssim t^{d+1}$.
Clearly, we may also assume that $\gamma(0)=0$.

Let $d\mu(x) = |\det (M^{\gamma,d}_{0})^{t} |\, \psi_{\lambda,\ell}( (M^{\gamma,d}_{0})^{t} x )dx$.
Then we have  $\widehat \mu (\lambda \gamma(t)) = \widehat{\psi_{\lambda,\ell}} (\lambda(\gamma_\circ^d(t) + (M_0^{\gamma,d})^{-1} \mathbf e(t))) $ by \eqref{eqn:taylor at 0}.
If $t < c \,\lambda^{-1/(d-\ell)}$ for a sufficiently small $ c$, $\lambda(\gamma_\circ^d(t) + (M_0^{\gamma,d})^{-1} \mathbf e(t)) \in S_\ell$.
Hence, it follows that
    \begin{align*}
        \int_0^1 |\widehat{\mu}(\lambda \gamma(t))|^{2} dt
        \ge  \int_0^{c \, \lambda^{-\frac1{d-\ell}}}\Big |\widehat{\psi_{\lambda,\ell}} \Big(\lambda(\gamma_\circ^d(t)
        + (M_0^{\gamma,d})^{-1} \mathbf e(t)) \Big)\Big|^{2} dt
        \gtrsim\,\lambda^{-\frac{1}{d-\ell}} .
    \end{align*}

    On the other hand,
    $
    I_{\alpha}(\mu) = \int |\widehat{\psi_{\lambda,\ell}}((M^{\gamma,d}_0)^{-1}\xi)|^{2}|\xi|^{\alpha-d}d\xi
    \lesssim \int_{ M^{\gamma,d}_0 S_\ell } |\xi|^{\alpha-d} d\xi
    $
    by the rapid decay of $\psi$.
    Hence,  we see
    \begin{align*}
            & I_{\alpha}(\mu) \leq C\int_{|\xi|\lesssim 1} |\xi|^{\alpha-d} d\xi
            + C \sum_{k=0}^{d-\ell-2}  \int_{\{\lambda^{1-\frac{k+2}{d-\ell}}
                \lesssim |\xi|\lesssim \lambda^{1-\frac{k+1}{d-\ell}}\}\cap M_0^{\gamma,d} S_{\ell}} |\xi|^{\alpha-d}
                d\xi.
              \end{align*}
              Using spherical coordinates,
              \begin{align*}
     &\int_{\{\lambda^{1-\frac{k+2}{d-\ell}}
                \lesssim |\xi|\lesssim \lambda^{1-\frac{k+1}{d-\ell}}\}\cap M_0^{\gamma,d} S_{\ell}} |\xi|^{\alpha-d}
                d\xi
               \lesssim  \Big({\int^{\lambda^{-1+\frac{k+1}{d-\ell}}}_{0}\!\!\cdots
                \!\!\int^{\lambda^{-1+\frac{k+1}{d-\ell}}}_{0}} d\theta_{d-1}\dots d\theta_{d-\ell-1}\Big)
                \times \\
           & \int^{\lambda^{-\frac{d-\ell-k-2}{d-\ell}}}_{0}
            \!\!\cdots
            \!\!\int^{\lambda^{-\frac{2}{d-\ell}}}_{0} \!\!\int^{\lambda^{-\frac{1}{d-\ell}}}_{0}
            \Big({\!\!\int^{1}_{0}
            \!\!\cdots
            \!\!\int^{1}_{0}}
                        \!\!\int^{\lambda^{1-\frac{k+1}{d-\ell}}}_{\lambda^{1-\frac{k+2}{d-\ell}}}
            \!\! r^{\alpha-1} dr d\theta_{d-\ell-2}   \cdots d\theta_{d-k-\ell-1} \Big) 
             \cdots  d\theta_{1}.
            \end{align*}
Hence, evaluating the integrals we get
              \begin{align*}
         I_\alpha(\mu)
         \lesssim \sum_{k=0}^{d-\ell-1} \lambda^{h(k)}, 
    \end{align*}
    where
    \[h(k)={(\alpha-\ell-1) - \frac{1}{d-\ell}\left( (k+1)(\alpha-\ell-1) + \frac{(d-\ell-k-2)(d-\ell-k-1)}{2} \right)} .\]

Clearly, \eqref{eqn:curve} implies $\lambda^{-\frac{1}{d-\ell}}\lesssim \lambda^{-\delta}\sum_{k=0}^{d-\ell-1} \lambda^{h(k)}$.  Letting
$\lambda\to \infty$ we get 
\[
\delta \le \frac1{d-\ell} + \max_{0\le k \le d-\ell-1} h(k).
\]
Since $h(x)$ attains the maximum at $x= d-\alpha-1/2$, it is easy to see that
    $
 \max_{0 \le k \le d-\ell-1 }  h(k)  =    h([d-\alpha]) .
    $
    Since $d-\ell-1 \ge [d-\alpha]$, we now consider the cases  $d-\ell-1 = [d-\alpha]$  and  $d-\ell-1 > [d-\alpha]$, separately.
    When $d-\ell-1 = [d-\alpha]$, we have
    \begin{align}\label{eqn:b}
    \delta \le \frac{1}{d-\ell} + h(d-\ell-1) = \frac{1}{[d-\alpha]+1}.
    \end{align}

    When $d-\ell-1 >[d-\alpha]$, we examine the value of $(d-\ell)^{-1} + h([d-\alpha])$ for $\ell = 0,\dots, d-[d-\alpha]-2$.
    Since  $(d-\ell)^{-1}+ h([d-\alpha])$ with $\ell = d-[d-\alpha]-2$ is the minimum, we get
    \begin{align}\label{eqn:c}
    \delta \le \frac{1}{d-(d-[d-\alpha]-2)} + h([d-\alpha]) = 1- \frac{d-\alpha}{[d-\alpha]+2}.
    \end{align}
    Thus we conclude that $\delta$ has upper bounds \eqref{eqn:b} or \eqref{eqn:c} for $d-[d-\alpha]-1 <\alpha \le d- [d-\alpha]$, which gives \eqref{sub:b}.
    Especially for $[d-\alpha]=0$ i.e. $d-1<\alpha < d$, the minimum value is $1-\frac{d-\alpha}{2}$, which is \eqref{sub:a}.

Finally, we show \eqref{sub:c}. In this case, $[d-\alpha] = d-1$, i.e. $0<\alpha\le 1$. Repeating the same argument, we see that \eqref{eqn:b} implies
$\delta \leq \frac1d$.   Hence it suffices to show
 $\delta \le \alpha$   for  $\alpha \in (0,d)$.
    To obtain this, let $\alpha_\ast \in (\alpha, d)$ and  consider
    $d\mu(x)=|x|^{-d+\frac{\alpha_{\ast}}{2}}\psi(x)dx$ for a Schwartz function $\psi$ in the above.
    It is easy to see $
        \widehat{\mu}(\xi)
        = C |\cdot|^{-\frac{\alpha_{\ast}}{2}}\ast\widehat{\psi}(\xi)
        \approx (1+|\xi|)^{-\frac{\alpha_{\ast}}{2}}$. So we get
    \[
    \int_0^1 |\widehat\mu(\lambda \gamma(t))|^2 dt \gtrsim \lambda^{-\alpha_\ast}.
    \]
    Since $\alpha < \alpha_\ast$,
    $
    I_\alpha(\mu) = \int |\widehat\mu (\xi)|^2 |\xi|^{\alpha -d} d\xi \le \int_{|\xi|>1}|\xi|^{\alpha-d-\alpha_\ast} d\xi + \int_{|\xi|<1} |\xi|^{\alpha-d} d\xi \lesssim 1.
       $
    Hence, \eqref{eqn:curve} implies  $\delta \le \alpha_\ast$ for any $\alpha_\ast\in(\alpha,d)$, which gives  $\delta \le \alpha$ as desired.
\end{proof}

Now we consider the lower bounds for $\kappa$ in Theorem
\ref{thm:L^2-L^q est}. We define the intervals $J_\circ(\ell)\subset
[1,\infty)$ by
\[
J_\circ (\ell) =
\begin{cases}
\, J(\ell), &\textrm{for }\, \ell = -1,0,\cdots, d-3-[d-\alpha],\\
\, [\,2 \beta_{d-\ell-1}(\alpha -\ell-1),\, 2 \beta_{d-\ell}(\alpha -\ell)\,],
&\textrm{for } \ell = d-2-[d-\alpha],\\
\, [\, 1,\,  2 \beta_{[d-\alpha]+1}(1-\langle d-\alpha\rangle) \,],
&\textrm{for } \ell = d-1-[d-\alpha].
\end{cases}
\]
For each $q \in J_\circ(\ell)$ we also define
$\kappa_\circ(\alpha,q,\ell)$ given by
\begin{align*}\kappa_\circ(\alpha,q,\ell) = 
\begin{cases}
\tfrac12-\tfrac{\alpha}{q}, & \text{if }  q\in J_\circ(-1),  \\
\tfrac 12  - \tfrac{\alpha-\ell}q + \tfrac1{d-\ell} \big(  \tfrac{\beta_{d-\ell}(\alpha-\ell)}q - \tfrac 12 \big), & \text{if }   q \in J_\circ(\ell),  
\end{cases}
\end{align*}
for $0 \le \ell \le d-1-[d-\alpha]$. 
Then
$\kappa_\circ(\alpha,q,\ell) = \kappa(\alpha, q,\ell)$ for $q \in J(\ell)$, $-1 \le
\ell \le d-3-[d-\alpha]$.
Also, for given $\alpha$ and $\ell$, $\kappa_\circ(\alpha,q,\ell)$ is defined only for $q\in J_\circ(\ell)$.
It is easy to see that $\kappa_\circ(\alpha,q,\ell)$ continuously decreases as $\ell$ increases.

\begin{proposition}\label{pro:nece}
    Suppose \eqref{eqn:l2q} holds with $\mu$, $\gamma$ and $g$ which are given as in Theorem \ref{thm:L^2-L^q est}.  For $q \in J_\circ(\ell)$,  
    \begin{align}\label{kapci}
    \kappa \geq  \kappa_\circ(\alpha,q,\ell).
    \end{align}
   In addition, $\kappa\ge (d-\alpha)/4$ when $d-1\leq \alpha\leq d$. 
    \end{proposition}

\begin{proof}[Proof of Proposition \ref{pro:nece}]
We show \eqref{kapci}   first.
Fix $\alpha$ and consider the measure $\mu_\circ$ given by
    \begin{align}\label{def:spec meas}
    d\mu_\circ(x) = \psi (x) \prod^{[d-\alpha]}_{j=1} d\delta(x_{j}) |x_{[d-\alpha]+1}|^{-\langle d-\alpha\rangle} dx_{[d-\alpha]+1} \cdots dx_{d} ,
    \end{align}
where $\psi$ is a smooth function supported in $B(0,1)$ and $\delta$
is the  delta measure. When $[d-\alpha]=0$, we write $d\mu_\circ(x) =
\psi(x) |x_1|^{-\langle d-\alpha \rangle} d x_1 dx_2\cdots d x_d$.
Then, as can be easily checked $\mu_\circ$ satisfies
\eqref{eqn:alpha-diml measure}.

Let $g(y):=\lambda^{-\frac{1}{2}}\chi_{\lambda\gamma(I)+O(1)}(y)$.
Then $ |\widehat g(x) | = \lambda^{-\frac12} \big|
\int_{\lambda\gamma(I) + O(1)} e^{i x\cdot y }dy \big| \gtrsim
\lambda^{\frac12} $ whenever $x\in B(0,c\lambda^{-1})$ for a
sufficiently small $c>0$. It follows that
\begin{align*}
    \|\widehat{g}\|_{L^{q}(d\mu)} \gtrsim \lambda^{\frac12}\mu(B(0,c\lambda^{-1}))^{\frac1q} \sim \lambda^{\frac12 - \frac\alpha
    q}.
\end{align*}
Since $\|g\|_{L^{2}(\mathbb{R}^{d})}\sim 1$, \eqref{eqn:l2q} and
letting $\lambda\to \infty$ gives
 $\kappa\geq 1/2- \alpha/ q$.

Now let $\ell$ be an integer such that $0 \le \ell \le
d-1-[d-\alpha]$. Let us consider the measure $\mu$ defined by $\int F(x)
d\mu = \int F( (M_0^{\gamma,d})^{-t}x) d\mu_\circ(x)$. Note that $d\mu$
is a compactly supported positive Borel measure satisfying
\eqref{eqn:alpha-diml measure}. Let  $J =
[0,\lambda^{-\frac{1}{d-\ell}}]$ and set 
$g(y) = \chi_{\lambda\gamma(J)+O(1)}(y).$
Then  $\|\widehat{g}\|_{L^q(d \mu)}^q = \int |\widehat{ g} (
(M^{\gamma,d}_{0})^{-t}x) |^q d\mu_\circ(x)$ and
    \[
    |\widehat{g}((M^{\gamma,d}_{0})^{-t} x)|
    = \Big|\int_{\lambda\gamma(J)+O(1)} e^{i x \cdot (M^{\gamma,d}_{0})^{-1}(y-\lambda\gamma(0))} dy \Big|.
    \]
Using Taylor's expansion in \eqref{eqn:taylor at 0} we see that
$(M^{\gamma,d}_{0})^{-1}(y-\lambda\gamma(0))$ is contained in
$\lambda\gamma_{\circ}^d(J) + O(1)$. Hence,
\[
|\widehat{g}((M^{\gamma,d}_{0})^{-t} x)| \gtrsim \lambda^{1-\frac{1}{d-\ell}}  \chi_{P_{\ell}}(x),
\]
where
    $
    P_{\ell} = [0, c\lambda^{\frac{1}{d-\ell}-1}]\times[0, c\lambda^{\frac{2}{d-\ell}-1}]\times\cdots\times[0, c\lambda^{\frac{d-\ell}{d-\ell}-1}]\times[0, c]\times\cdots\times[0, c],
    $
for a small $c>0$. Since $\mu_\circ(P_{\ell})\sim
\lambda^{-(\alpha-\ell)+\frac{\beta_{d-\ell}(\alpha-\ell)}{d-\ell}}$,
we get
\begin{align*}
\|\widehat{g}\|_{L^q(d\mu)}
\gtrsim \lambda^{1-\frac{1}{d-\ell}}\Big(\int \chi_{P_\ell}(x)d\mu_\circ(x)\Big)^{\frac 1q}
\sim \lambda^{1-\frac{\alpha-\ell}{q} + \frac{1}{d-\ell}(\frac{\beta_{d-\ell}(\alpha-\ell)}{q}-1)} .
\end{align*}
Combined with this and $\|g \|_{L^{2}} \sim
\lambda^{\frac{1}{2}-\frac{1}{2(d-\ell)}}$, \eqref{eqn:l2q} gives, for $0 \le \ell \le d-1-[d-\alpha]$,
\begin{align*}
\kappa \geq \frac 12  - \frac{ \alpha - \ell } q  + \frac 1 {d-\ell} \left( \frac{\beta_{d-\ell}(\alpha-\ell)}{q} -\frac 12 \right). 
\end{align*}

Considering the maximum along $\ell$ and the lower bound
$\kappa \ge \frac 12 -\frac \alpha q$, we can see that $\kappa \ge
\frac 12 -\frac \alpha q$ for $q \in J_\circ(-1)$, i.e. $q \ge 2
\beta_d(\alpha)$. When $2 \beta_{d-1}(\alpha-1) \le q \le 2
\beta_{d}(\alpha)$, i.e.  $q \in J_\circ(0)$, we get $\kappa \ge  \frac 12 - \frac{ \alpha } q  + \frac 1 {d } \big( \frac{\beta_{d }(\alpha)}{q} -\frac 12 \big)$. 
Similarly for each $\ell$, we conclude that $\kappa \ge
\kappa_\circ(\alpha,q,\ell)$ for $q \in J_\circ(\ell)$.

We now show that $\kappa \ge (d-\alpha)/4$ when $d-1\leq \alpha \leq d$.  
For this, we adapt the argument in \cite{Er}.  
Let $G_1$ be a Schwartz function supported in $ D := [0,\lambda^{\frac{1}{2}}]\times[0,1]\times\cdots\times[0,1]$ $\subset \lambda \gamma_{\circ}^{d}(I) + O(1)$ such that  $\|G_1\|_{L^2} =1$ and $|\widehat{G_{1}}(x)| > \lambda^{\frac{1}{4}}/100$ on a rectangle $D^\ast$ of dimension $\lambda^{-\frac12}\times 1\times\cdots\times1$. 

For a fixed $\lambda \ge 1$, we set $T = \lambda^{\frac{\alpha-(d-1)}{2}} $ and define a Schwartz function $G_2$ by 
\begin{align*}
\widehat{G_{2}}(x) := T^{-\frac{1}{2}} \sum^{T-1}_{k=0} \widehat{G_{1}}(x-\frac{k}{T}e_{1}),
\end{align*}
where $e_1 = (1,0,\dots,0)\in \mathbb R^d$.
Then $|\widehat{G_{2}}|\gtrsim T^{-\frac{1}{2}}\lambda^{\frac{1}{4}}$ on the set
$S:=\bigcup^{T-1}_{k=0}(D^* + \frac{k}{T}e_{1})$ and 
$\|G_{2}\|_{2}^2 =  T^{-1} \sum^{T-1}_{k=0} \| \widehat{G_{1}}(\cdot -\frac{k}{T}e_{1})\|_2^2 =1$.
Moreover $G_{2}$ is supported in $D$.  Hence, if we set \[G_{3}(x) := |\det M^{\gamma,d}_{0}|^{-\frac{1}{2}}  G_{2} ((M^{\gamma,d}_{0})^{-1}x),\]
then $G_{3}$ is supported in $M_{0}^{\gamma,d}D \subset \lambda \gamma(I) +O(1)$ and $\| {G_{3}}\|_{L^{2}} = 1$.

Let us set $ d\mu_\circ(x) = \lambda^{\frac{d-\alpha}{2}} \chi_{S}(x) dx$. It is not difficult to  verify that $\mu_\circ$ satisfies  \eqref{eqn:alpha-diml measure}. In fact, if $ \lambda^{-\frac{1}{2}} \leq \rho < 1 $, there exists an integer $j$ such that $j/T \le \rho \le (j+1)/T$ by the definition of $S$.  Hence, for any $x \in \mathbb R^d$, we have 
\begin{align*}
\mu_\circ(B(x ,\rho))
=   \lambda^{\frac{d-\alpha}{2}} |S\cap B(x ,\rho)|
\lesssim \lambda^{\frac{d-\alpha}{2}} (j+1)\lambda^{-\frac12} \rho^{d-1} 
\lesssim  \lambda^{-\frac{\alpha - (d-1)}{2}} T \rho^{d}
\le \rho^d \le \rho^\alpha.
\end{align*}

The other cases $0 < \rho < \lambda^{-\frac{1}{2}}$ and $\rho \geq 1$ can be handled similarly. So,  by  Lemma \ref{lem:rescale} {the measure $\mu$} defined by \[\int F(x) d\mu = \int F( (M_0^{\gamma,d})^{-t}x) d\mu_\circ(x)\] also satisfies   \eqref{eqn:alpha-diml measure}.  Since $T \le \lambda^{\frac{\alpha-(d-1)}{2}} < T+1$,  it follows that 

\begin{align*}
\|{\widehat{G_{3}}}\|_{L^{q}(d \mu)}^q
=  \int |\widehat{G_{3}}((M^{\gamma,d}_{0})^{-t}x)|^{q} d\mu_\circ(x)  
\sim  \int |\widehat{G_{2}}(x)|^{q} d\mu_\circ(x)  
\gtrsim T^{-\frac{q}{2}}\lambda^{\frac{q}{4}} \lambda^{\frac{d-\alpha}{2 }}|S| 
\gtrsim \lambda^{\frac{q(d-\alpha)}{4}}.
\end{align*}

Hence we see $\kappa\geq (d-\alpha)/4$ by letting $\lambda\rightarrow \infty$.
\end{proof}


\section{Proof of Theorem \ref{thm:osc-est} and \ref{thm:osc-est2}}\label{sec:pf multi}

For a given $\alpha$, let $\ell$ be an integer in $[0,
d-1-[d-\alpha]]$. As $\ell$ increases, the oscillatory decay in
\eqref{eqn:osc1} gets worse while the range \eqref{eqn:range} gets
wider. The case of $\ell=0$ is already established in \cite{HaLe}.
To show Theorem \ref{thm:osc-est} for the other cases, we consider
the collection $\Gamma (k,\epsilon)$ of curves which is given by
\[
\Gamma (k,\epsilon)=\Big\{\gamma \in C^{d+1}(I) :
\|\gamma-\gamma_\circ^k \|_{ C^{k+1}(I)}\le \epsilon
\Big\},
\]
where
\begin{equation}\label{gamma-k}\gamma_\circ^k(t)=\Big(t,\,{t^2}/{2!},\,\dots,\, {t^k}/{k!},\,0,\,\dots,\, 0 \Big),\,\, 1 \le k \le d.\end{equation} The
curves in  $\Gamma (k,\epsilon)$ are nondegenerate in $\mathbb R^k$
when they are projected to $\mathbb R^k\times\{0\}$. Viewing these curves as
nondegenerate curves in $\mathbb R^k$ provides various multilinear estimates
under a  separation condition between functions (see Proposition \ref{lem:l2}). From these multilinear estimates we can
obtain the linear estimate by adapting the argument in \cite{HaLe}.
The difference here is that we run induction on scaling argument on
each $k$-linear estimates which were not exploited before. This
requires control of rescaling of measures when $d-k = \ell$
variables are fixed.

\subsection{\it Normalization of curves}
In Lemma \ref{lem:normalization} we show that any nondegenerate curve defined in a sufficiently small interval can be made arbitrarily close to  $\gamma_\circ^k$.
This can be shown by Talyor expansion of $\gamma$ of degree $k$ and rescaling.
It is worth noting that the condition \eqref{eqn:non-degenerate} does not guarantee $| \det M_\tau^{\gamma, k} | \ge c >0$ for some $c$, where
\begin{equation}\label{eqn:M}
 M_\tau^{\gamma, k}= (
\gamma'(\tau),\gamma''(\tau),\cdots,\gamma^{(k)}(\tau),e_{k+1},\dots, e_d  )
\end{equation}
and $e_j$'s are the unit vectors whose $j$-th component is 1. However,  by Lemma 2.1 in \cite{HaLe}, we may assume that  (after a finite number of decompositions and rescaling) any non-degenerate curve $\gamma$ is close to $\gamma_\circ^d$ in a small interval.
Using this we can see that $ M_\tau^{\gamma, k}  $ is invertible and there is a constant $B>0$ by which $\|  (M_\tau^{\gamma, k})^{-1} \| $ is uniformly bounded for $\tau \in I$.
(Here $\|  M\|$ denotes the usual matrix norm such that $\|M\| = \max_{|x| = 1} |Mx|$.)
    In fact, if $\gamma \in \Gamma  (d, \epsilon )$, we  have $\gamma = \gamma_\circ^{d} + \mathbf e_d $ such that $\| \mathbf e_d \|_{C^{d+1}(I)} < \epsilon$.
    Then $\det M_\tau^{\gamma, k} = \det(\gamma'_\circ ,\gamma''_\circ,\dots,\gamma^{(k)}_\circ,e_{k+1},\dots,e_d) + \textit{error terms}$.
    For sufficiently small $\epsilon$, it follows that $\det M_\tau^{\gamma, k} \ge \frac12$. (Note that $\det(\gamma'_\circ ,\gamma''_\circ,\dots,\gamma^{(k)}_\circ,e_{k+1},\dots,e_d) =1$.)

For $a,b\in \mathbb R$, $a\neq b$, let us set
\[ |\![a,b]\!|=\begin{cases} \,\, [a,b] \,\, \text{ if } a<b,\\ \,\, [b,a] \,\, \text{ if } b<a. \end{cases}\]

We define the normalized curve by setting
\begin{equation}
\label{eqn:normalcur} \gth(t)= (M_\tau^{\gamma,k}\,D_h^k)^{-1
}(\gamma(ht+\tau)-\gamma(\tau)),
\end{equation}
where $D_h^k$ is the diagonal matrix given by
$D_h^k =(he_1, h^2e_2, \dots, h^{k} e_k, e_{k+1},\dots, e_d  )$.
Then $\gth(t)$ can be close to $\gamma_\circ^k$ if $h$ is sufficiently small, as follows.

\begin{lemma} \label{lem:normalization}
    Let $\tau \in I$ and  $\gamma\in \Gamma ( d,\epsilon)$ for some $\epsilon >0$.
    Then,  there is a constant $\delta>0$ such that
    $\gamma_{\tau}^h\in \Gamma (k,\epsilon)$ whenever
    $|\![\tau,\tau+h]\!| \subset I$,  $0<|h|\le \delta$.
\end{lemma}
\begin{proof}
    We may assume that $h >0$, i.e. $ |\![\tau,\tau+h]\!| = [\tau,\tau+h]$.
The case that $h <0$ can be shown in the same manner.
    By Taylor's expansion, we have
    \begin{align*}
    \gamma (ht +\tau) - \gamma(\tau)
    & = \gamma'(\tau) ht + \gamma''(\tau) \frac{(ht)^2}{2!} + \cdots + \gamma^{(k)}(\tau) \frac{(ht)^k}{k!} + \mathbf e (\tau, h, t) \\
    & = M_\tau^{\gamma, k} D_h^k \gamma_\circ^k(t) + \mathbf e (\tau, h, t),
    \end{align*}
    where $\| \mathbf e (\tau, h, t) \|_{C^{k+1}(I)} \le C h^{k+1}$ for some constant $C>0$ independent of $\tau$.
Hence we obtain $\| \gamma_\tau^h - \gamma_\circ^k \|_{C^{k+1}(I)}  = \| (M_\tau^{\gamma, k} D_h^k)^{-1} \mathbf e(\tau, h,t)\|_{C^{k+1}(I)} \lesssim  h$, which implies that $\gamma_\tau^h \in \Gamma (k,\epsilon)$ if we take $\delta \lesssim \epsilon/2$.
\end{proof}

\subsection{\it Rescaling of measure}

For $M > 0$, we denote by $\mathfrak M(\alpha,M)$ the set of
compactly supported positive Borel measures satisfying $0<\langle
\mu \rangle_\alpha \le M$. Let $\mu \in \mathfrak M (\alpha,M)$, and
let $A$ be a non-singular matrix.  Let us now define a measure
$\mu^k_{A,h}$ by setting
\begin{equation}\label{eqn:normalmeasure}
\int F(x) d \mu^k_{A,h} (x) =\int F( D_h^k A x) d \mu (x)
\end{equation}
for any compactly supported continuous function $F$ and $0 < |h| < 1$.  
By the Riesz representation theorem we see that $\mu^k_{A,h}$ is the unique measure given by \eqref{eqn:normalmeasure}.

\begin{lemma}\label{lem:rescale}
    Let $\mu$ and $A$ be given as above, $\ell = 0,1,\dots, d-1-[d-\alpha]$.  Set $k = d-\ell$. Then,  $\mu^k_{A,h}$ is also a Borel measure satisfying
    \begin{equation}\label{eqn:decom}
    \langle \mu^k_{A,h} \rangle_\alpha \leq C \langle \mu\rangle_\alpha \|A^{-1} \|^\alpha
    |h|^{-\beta_k(\alpha-d+k)}.
    \end{equation}
    Here $C$ is independent of $h, A$.
\end{lemma}

\begin{proof}
       By the proof of Lemma 2.3 in \cite{HaLe}, it suffices to show that
    \[
\mu_h^k (B(0,\rho)) \le C \langle \mu\rangle_\alpha  |h|^{-\beta_k(\alpha-d+k)} \rho^\alpha,
    \]
    where $\mu_h^k : = \mu_{I_d,h}^k$ and $I_d$ is the $d\times d $ identity matrix.
    It is clear that  $\mu_h^k(B(0,\rho))  = \mu( (D_h^k)^{-1}B(0,\rho))\le \mu ( \mathcal R)$, where
    $\mathcal R$ is a rectangle of dimension $|h|^{-1} \rho\times \cdots \times |h|^{-k}\rho \times \rho \times \cdots \times \rho$.
    If we denote by $\widetilde{\mathcal R}$ a larger rectangle of dimension $|h|^{-1} \rho\times \cdots \times |h|^{-k}\rho \times |h|^{-([d-\alpha]+1)}\rho \times \cdots \times |h|^{-([d-\alpha]+1)}\rho$ which contains $\mathcal R$, then it follows that $\mu(\widetilde{\mathcal R}) \sim |h|^{-([d-\alpha]+1)(d-k)} \mu(\mathcal R)$.
    Since $1 \le [d-\alpha]+1 \le k$, $\widetilde{\mathcal R}$ is covered by cubes $Q_1,\dots, Q_N$ of side length $|h|^{-([d-\alpha]+1)} \rho$ with $N \lesssim |h|^{-(k-1-[d-\alpha])(k-[d-\alpha])/2}$.
 Since $\mu(Q_i)\le  \langle \mu\rangle_\alpha |h|^{-\alpha([d-\alpha]+1)} \rho^\alpha$,   we get
    \begin{align*}
    &\mu_h^k(B(0,\rho)) \lesssim |h|^{([d-\alpha]+1)(d-k)} \mu(\widetilde{\mathcal R})
         \le |h|^{([d-\alpha]+1)(d-k)} \sum_{i=1}^N \mu(Q_i)
    \le \langle \mu\rangle_\alpha |h|^{-\beta_k(\alpha - d + k)} \rho^\alpha.
    \end{align*}
    This completes the proof.
\end{proof}

\subsection{\it Multilinear ($k$-linear) estimates}
Let us set, for $\lambda \ge 1$,
\begin{align*}
\mathcal E^{\gamma}_{\lambda}f(x) = a(x) \int_{I} e^{i\lambda x \cdot \gamma(t)} f(t) dt ,
\end{align*}
where $a$ is a bounded function supported in $B(0,1)$ with
$\|a\|_\infty\le 1$. As mentioned above, we need to prove $k$-linear
estimates for $\mathcal E^{\gamma}_{\lambda}$ while $\gamma \in \Gamma
(k,\epsilon)$. 
This can be achieved simply by freezing other $d-k$
variables. By applying Lemma 2.5 in \cite{HaLe} and Plancherel's
theorem, we obtain a $k$-linear $L^2\rightarrow L^2$ estimate.

\begin{lemma}\label{lem:l2} Let $\gamma \in \Gamma (k,\epsilon)$ and $\mathcal I_1,\dots, \mathcal I_k$
    be closed intervals  contained in $I$ which satisfy $\min_{i\neq
        j}\dist(\mathcal I_i, \mathcal I_j)\ge L$. If $\epsilon>0$ is
    sufficiently small, then there is a constant $C$, independent of
    $\gamma$, such that
    \begin{equation}\label{eqn:k-l2}
    \| \prod_{i=1}^k \mathcal E_\lambda^\gamma f_i \|_{L^2(\mathbb R^d)}\le CL^{-\frac{k^2-k}{4}}\lambda^{-\frac k 2}   \prod_{i=1}^k  \| f_i \|_{L^2(\mathbb R)}
    \end{equation}
    whenever $f_i$ is supported in $\mathcal I_i$, $i=1,2,\dots, k$.
\end{lemma}

\begin{proof}For the proof, it suffices to show that for a constant vector $\mathbf c \in \mathbb R^{d-k}$,
    \begin{equation}\label{eqn:freezing}
    \int \Big| \prod_{i=1}^k \mathcal E_\lambda^\gamma f_i (x_1,\dots,x_k,\mathbf c) \Big|^2 d x_1 \cdots d x_k
    \le C L^{-\frac{k^2-k}{2}}\lambda^{- k }   \prod_{i=1}^k  \| f_i \|^2_{L^2}.
    \end{equation}
    Then \eqref{eqn:k-l2} follows by integrating along $\mathbf c$ since $a$ is supported in $B(0,1)$ and $\|a\|_\infty \le    1$.      
    To prove this, let us set $\gamma (t) = (\gamma_{\star}(t),\gamma_{\mathbf c}(t))$ where $\gamma_{\star}(t)$ is the first $k$ components of $\gamma(t)$ and the rest of the components are denoted by $\gamma_{\mathbf c}(t)$. Also let us set
    $$F(\mathbf t) = e^{i \lambda \mathbf c \cdot \sum_{i=1}^k \gamma_{\mathbf c}(t_i)} \prod_{i=1}^k f_i(t_i)$$
    where $\mathbf t = (t_1,\dots,t_k)$. Then we have
    \[
\prod_{i=1}^k \mathcal E_\lambda^\gamma f_i (x_1,\dots,x_k,\mathbf c)
    = \int_{I^k} e^{ i \lambda (x_1,\dots,x_k)\cdot \sum_{i=1}^k\gamma_{\star}(t_i) } F(\mathbf t) d \mathbf t.
    \]
Since $\gamma \in \Gamma (k, \epsilon)$, we have $\gamma_\star(t) = (t,
t^2/2!,\dots,t^k/k!) + \mathbf e$ such that $\|\mathbf
e\|_{C^{k+1}(I)} \le \epsilon$. Then we can apply  Lemma 2.5 in
\cite{HaLe}, to say $k$-linear estimates in $\mathbb R^k$, or more
directly change of variables and Plancherel's theorem. Since
$\|F\|^2_{L^2} =  \prod_{i=1}^k  \| f_i \|^2_{L^2}$ we get
\eqref{eqn:freezing}.
\end{proof}

Now we obtain an $L^p \rightarrow L^q(d\mu)$ estimate by interpolating \eqref{eqn:k-l2} with the trivial $L^1 \rightarrow  L^\infty(d\mu) $ estimate.

\begin{proposition}\label{pro:l2fractal}
    Let $\mathcal I_1,\dots, \mathcal I_k$, and $\gamma \in \Gamma (k,\epsilon)$ be given as in \emph{Lemma \ref{lem:l2}}.
    Suppose  $\mu \in \mathfrak M(\alpha,1)$. If $\epsilon>0$ is
    sufficiently small, then for $1/p+1/q\le 1$ and $q\ge 2$ there is a
    constant $C$, independent of $\gamma$, such that
    \[
    \| \prod_{i=1}^k \mathcal E_\lambda^\gamma f_i \|_{L^q(d\mu)}\le C \langle \mu\rangle_\alpha^{\,\,\frac1{q}}
    L^{-\frac{k^2-k}{2q}}\lambda^{-\frac{\alpha-d+k}{q}}\prod_{i=1}^k\| f_i \|_p
    \]
    whenever $f_i$ is supported in $\mathcal I_i$, $i=1,2,\dots, k$.
\end{proposition}

\begin{proof}
Since we have the trivial estimate $\| \prod_{i=1}^k \mathcal E_\lambda^\gamma f_i \|_{L^\infty(d\mu)} \le \prod_{i=1}^k \|f_i\|_{L^1}$,
  in view of interpolation  it suffices to show that
    \[
    \| \prod_{i=1}^k \mathcal E_\lambda^\gamma f_i \|_{L^2(d\mu)} \le C \langle \mu\rangle_\alpha^{\,\,\frac12}
    L^{-\frac{k^2 -k}{4}} \lambda^{- \frac{\alpha - d +k}{2}} \prod_{i=1}^k \|f_i\|_{L^2}.
    \]
    Since the Fourier transform of $ \prod_{i=1}^k \mathcal E_\lambda^\gamma f_i$ is
    supported in a ball of radius $C \sqrt{2k} \lambda$ for some constant $C> 0$, we observe  that
    $
    \prod_{i=1}^k \mathcal E_\lambda^\gamma f_i = ( \prod_{i=1}^k \mathcal E_\lambda^\gamma f_i) \ast \phi_\lambda,
    $
    where $\phi_\lambda(x) = \lambda^d \phi(\lambda x)$ and $\phi$ is a Schwartz function such that
    $\widehat\phi =0 $ if $|\xi| \ge 2 C \sqrt{2k}$, and $\widehat \phi =1$ if $|\xi |\le C \sqrt{2k}$.
    Note that $ |\phi_\lambda| \ast \mu(x) \le C \langle \mu \rangle_\alpha \lambda^{d-\alpha}$.
    By Lemma \ref{lem:l2}, it follows that
    \[
    \| \prod_{i=1}^k \mathcal E_\lambda^\gamma f_i \|_{L^2(d\mu)}
    \le \| \prod_{i=1}^k \mathcal E_\lambda^\gamma f_i \|_{L^2(\mathbb R^d)} \| | \phi_\lambda | \ast \mu\|^{\frac12}_{\infty} \le C \langle \mu \rangle_\alpha^{\,\,\frac12} L^{-\frac {k^2 - k}{4}} \lambda^{-\frac{\alpha - d +k}{2}} \prod_{i=1}^k \| f_i \|_{L^2}
    \]
    as desired.
\end{proof}

\renewcommand{\clag}{\Gamma (k,\epsilon)}

\subsection{\it The induction quantity}
For $\lambda \geq 1$, $1\le p, q\le \infty$, and $\epsilon>0$, we
define $Q_\lambda  =Q_\lambda(p, q,\epsilon)$ by setting
\begin{align}\label{eqn:aqlambda}
Q_\lambda  = \sup\{\, \|\mathcal E^\gamma_\lambda f\|_{L^q(d\mu)}: \mu\in
\mathfrak M(\alpha,1),\, \gamma\in \clag, \, {\|f\|_{L^p(I)}\le 1},\,  a \in \mathfrak A \},
\end{align}
where $\mathfrak{A}$ is a set of measurable functions supported in
$B(0,1)$ and $\|a\|_\infty\le 1$. It is clear that $Q_\lambda$ is
finite for any $\lambda
>0$.

\begin{lemma}\label{lem:induction}
    Let $\gamma\in  \clag$, $\mu\in \mathfrak M(\alpha,1)$,
    and let $\lambda\ge  1$, $0<|h|<1$. Suppose that $f$
    is supported in the interval $|\![\tau,\tau+h]\!|\subset[0,1]$. Then, if
    $\epsilon>0$ is sufficiently small,  there is a constant $\delta>0$,
    independent of $\gamma$, such that if $0<|h|\le \delta$
    \begin{equation}\label{eqn:qlambda}
    \|\mathcal E^\gamma_\lambda f \|_{L^q(d\mu)}\le C \langle \mu \rangle_\alpha^{\,\,\frac1q} \,|h|^{1-\frac1p-\frac{\beta_{k}(\alpha-d+k)}{q}} Q_{\lambda }
    \|f\|_p.
    \end{equation}
\end{lemma}

\begin{proof}
    Let us denote $f_h(t) = h f(ht +\tau)$.
    Recalling \eqref{eqn:normalcur} we have 
    \begin{align*}
    | \mathcal E_\lambda^\gamma f(x) |  = \Big|\int_I e^{ i \lambda x \cdot (\gamma(ht + \tau) - \gamma(\tau))} a(x) f_h(t) dt \Big|
    = \Big|\int_I e^{i \lambda (M_\tau^{\gamma, k} D_h^k)^t x \cdot \gamma_\tau^h(t)}  a(x) f_h(t) dt \Big|.
    \end{align*}
    Let us set $\mu_{\tau,h}^k := \mu_{(M_\tau^{\gamma, k})^t,h}^k$ which is given by \eqref{eqn:normalmeasure}.
    Assuming that $\langle \mu\rangle_\alpha \neq 0$, we set
    \begin{equation*}
    d \widetilde \mu(x) = \frac { |h|^{\beta_k( \alpha -d+k)}}{C \| (M_\tau^{\gamma, k})^t \|^\alpha  \langle \mu \rangle_\alpha } d \mu_{\tau,h}^k(x).
    \end{equation*}
    Then $\langle \widetilde \mu \rangle_\alpha \le 1$, i.e.
    $\widetilde \mu \in \mathfrak M(\alpha,1)$ by Lemma \ref{lem:rescale}. 
Routine changes of variables gives
    \begin{align*}
    \| \mathcal E_\lambda^\gamma f \|_{L^q(d\mu)}^q
    & \le \int \Big| a_{\tau,h}^k (x)  \int_I e^{ i \lambda x \cdot \gamma_\tau^h(t)}  f_h(t) dt \Big|^q d \mu_{\tau,h}^k (x) \\
    & = \frac{ C \| (M_\tau^{\gamma, k})^t \|^\alpha \langle \mu \rangle_\alpha   }{ |h|^{  \beta_k(\alpha-d+k)} }
    \int \Big| a_{\tau,h}^k (x) \int_I e^{ i \lambda x \cdot \gamma_\tau^h(t)}  
    f_h(t) dt \Big|^q d \widetilde\mu (x) ,
    \end{align*}
    where $a_{\tau,h}^k (x) = a ((M_\tau^{\gamma, k} D_h^k)^{-t}x)$.
    If $\epsilon >0$ is sufficiently small, then $\| (M_\tau^{\gamma, k})^{-t} \| \le c $ uniformly for $\gamma \in \Gamma (k, \epsilon)$.
    Then $\gamma_\tau^h(t) \in \Gamma (k, c |h|\epsilon) \subset \Gamma (k,\epsilon)$ if $0 < |h| \le \delta$ for small $\delta=\delta(\epsilon)$.
    In addition, $a_{\tau,h}^k \in \mathfrak A$ since $\supp a_{\tau,h}^k = D_h^k (M_\tau^{\gamma, k})^t \supp a$.
   By the definition of $Q_\lambda $, it follows that
    \begin{align*}
    \int | \mathcal E_\lambda^\gamma f |^q d \mu (x)
    & \lesssim \langle \mu \rangle_\alpha |h|^{-\beta_k( \alpha -d +k)} \int |\mathcal E_\lambda^{\gamma_\tau^h} f_h |^q d \widetilde \mu(x)
    \le C \langle \mu \rangle_\alpha |h|^{-\beta_k( \alpha -d +k)} (Q_\lambda    \| f_h\|_p)^q,
    \end{align*}
    which implies \eqref{eqn:qlambda} as $\| f_h \|_p = h^{1-1/p} \|f\|_p$.
\end{proof}

\subsection{\it Proof of Theorem \ref{thm:osc-est}} Let $\ell$ be a fixed integer such that $1 \le \ell \le
d-1-[d-\alpha]$ and let $k = d-\ell$. We choose a sufficiently small 
$\epsilon>0$ such that $\det (M_\tau^{\gamma, k}) \ge \frac12$ if $\gamma \in
\Gamma(d,\epsilon)$, and Lemma \ref{lem:l2} and \ref{pro:l2fractal}
hold whenever $\gamma \in \Gamma(k,\epsilon)$. Let us be given  a
curve $\gamma \in C^{d+1}([0,1])$  satisfying
\eqref{eqn:non-degenerate}. By Lemma \ref{lem:normalization}, there
exists $\delta>0 $ such that $\gamma_\tau^h \in \Gamma(k,\epsilon)$
for $|h|<\delta$. Then Lemma \ref{lem:induction} also holds for such
$\gamma_\tau^h \in \Gamma(k,\epsilon)$. Thus, after decomposing the
interval $I$ into finite union of intervals of length less than
$\delta$,  by  rescaling  we may assume that $\gamma \in \Gamma (k,
\epsilon)$ and $\mu \in \mathfrak M (\alpha, 1)$.

In fact, we decompose $I = \bigcup_{j=0}^{n-1} [\frac j n,
\frac{j+1} n ] =: \bigcup_{j=1}^{n-1} I_j$ with $h:= 1/n<\delta$.
Then we have
\begin{align*}
\|\mathcal E_\lambda^\gamma f\|_{L^q(d\mu)} \le \sum_{j=0}^{n-1}
\| \mathcal E_\lambda^\gamma f\chi_{[jh, jh+h ]} \|_{L^q(d\mu)} = \sum_{j=0}^{n-1} (C_{\gamma,j,h})^{\frac1q} \| \mathcal E_\lambda^{\gamma_j} f_j \|_{L^q(d\mu_j)},
\end{align*}
where $f_j(t) = h f(ht +j h)\chi_I(t)$, $\gamma_j = \gamma_{jh}^h$,
and $ \mu_j = \frac 1 {C_{\gamma,j,h}} \mu_{jh}^h$ with 
$C_{\gamma,j,h} = C \| (M_{jh}^{\gamma,k} )^{-t} \|^\alpha
\langle\mu\rangle_\alpha$ $h^{-\beta_{d-\ell}(\alpha-\ell)}. $
Hence, it is enough to obtain the desired estimate for each $\|
\mathcal E_\lambda^{\gamma_j} f_j \|_{L^q(d\mu_j)}$. Clearly, from
Lemma \ref{lem:normalization}  and \ref{lem:rescale} it follows
that $\gamma_j \in \Gamma (k, \epsilon)$, and also $\mu_j \in
\mathfrak M(\alpha,1)$. Therefore we are reduced to showing
\eqref{eqn:osc1} for $\gamma\in \Gamma (k,\epsilon)$, $\mu\in
\mathfrak M(\alpha,1)$.

Let $q \ge p \ge 1$ be numbers satisfying the conditions in Theorem
\ref{thm:osc-est}.  Note that the other case $1\le q <p$ follows by
H\"{o}lder's inequality.  Also let $Q_\lambda  =Q_\lambda(p,
q,\epsilon)$ be defined by \eqref{eqn:aqlambda}. Then, for the proof
of Theorem \ref{thm:osc-est} we need to show
\begin{equation}\label{qq}
Q_\lambda  \lesssim \lambda^{-\frac {\alpha-\ell} q}.
\end{equation}

Let $\gamma\in \Gamma (k,\epsilon)$, $\mu\in \mathfrak M(\alpha,1)$
be given, and  $f$ be a function supported in $I$ with
$\|f\|_{L^p(I)}=1$ such that
\begin{equation}\label{qqq}
Q_\lambda =Q_\lambda(  p,q,\epsilon) \le  2 \| \mathcal E_\lambda^\gamma
f\|_{L^q(d\mu)}.
\end{equation}
 Let
$A_1,\dots, A_{k-1}$ be dyadic numbers such that
\[ 1=A_0\gg  A_1\gg A_2 \dots\gg  A_{k-1}.\]
These numbers are to be chosen later.  For $i=1, \dots, k-1$, let us
denote by $ \{ \cuxi^i\}$ the collection of closed dyadic intervals
of length $A_i$ which are contained in $[0,1]$. And we set
$f_{\cuxi^i}=\chi_{\cuxi^i} f$ so that, for each $i=1, \dots, k-1,$
$ f=\sum_{\cuxi^i} f_{\cuxi^i}$ almost everywhere whenever $f$ is
supported in $I$. Hence, it follows that
\begin{equation}\label{eqn:iscale}
\mathcal E^\gamma_\lambda f =\sum_{\cuxi^i} \mathcal
E^\gamma_\lambda f_{\cuxi^i}, \,\, i=1, \dots, k-1.
\end{equation}

We now recall the multilinear decomposition  from  \cite{HaLe} (Lemma 2.8).

\begin{lemma}\label{lem:multidecomp} Let $\gamma:I\to \mathbb R^d$ be a smooth
    curve. Let $A_0,A_1,\dots, A_{k-1}$, and $\{\cuxi^i\}$, $i=1, \dots,
    k-1$ be defined as in the above. Then, for any $x\in \mathbb R^d$,
    there is a constant $C$, independent of $\gamma, x$, $A_0,A_1,\dots,
    A_{k-1}$, such that
    \begin{equation}
    \begin{aligned}\label{eqn:1d-decomp}
    |\mathcal E^\gamma_\lambda & f(x)|\le C \sum_{i=1}^{k-1} A_{i-1}^{-2(i-1)} \max_{\cu i}
    |\mathcal E^\gamma_\lambda \ef{i}(x)|\\
    &+ CA_{k-1}^{-2(k-1)}\max_{\substack{\cuu1{k-1},\cuu2{k-1},
            \dots,\cuu k{k-1};\\ \Delta(\cuu1{k-1}, \cuu2{k-1},
            \dots,\cuu k{k-1})\ge A_{k-1}}} |\prod_{i=1}^k \mathcal E^\gamma_\lambda \eff{k-1}i
    (x)|^\frac1k.
    \end{aligned}
    \end{equation}
    Here $\cuu{i}j$ denotes the element in $\{\cuxi^i\}$ and $
\Delta(\cuu1{k-1},
            \dots,\cuu k{k-1})=\min_{1\le j<m \le k}\, \dist(\cuu j{k-1}, \cuu m{k-1}).
$
\end{lemma}

We consider the linear and multi-linear terms in
\eqref{eqn:1d-decomp}, separately. For the linear term, using Lemma
\ref{lem:induction} we see that
\begin{align*}
\Big\|\max_{\cu i} &|\mathcal E^\gamma_\lambda
\ef{i}|\Big\|_{L^q(d\mu)} \le \Big(\sum_{\cu i} \Big\|\mathcal
E^\gamma_\lambda \ef{i}\Big\|_{L^q(d\mu)}^q\Big)^\frac1q
\le
{A_i}^{1-\frac1p-\frac{\beta_{d-\ell}(\alpha -\ell)}{q}} Q_{ \lambda} \Big( \sum_{\cu i} \|\ef{i}\|_p^q\Big)^\frac1q  \\
&\le
{A_i}^{1-\frac1p-\frac{\beta_{d-\ell}(\alpha-\ell)}{q}} Q_{ \lambda}  \Big(
\sum_{\cu i} \|\ef{i}\|_p^p\Big)^\frac1p
={A_i}^{1-\frac1p-\frac{\beta_{d-\ell}(\alpha-\ell)}{q}} Q_{ \lambda}
\|f\|_p\,,
\end{align*}
because $\ell^p \subset \ell^q$ for $q \ge p$.
Applying Proposition
\ref{pro:l2fractal} to the multilinear term, we obtain
\[
\Big\| \max_{\substack{\cuu{d-\ell-1}1,\cuu{d-\ell-1}2, \dots,\cuu{d-\ell-1}{d-\ell};\\
        \Delta(\cuu{d-\ell-1}1, \cuu{d-\ell-1}2, \dots,\cuu{d-\ell-1}{d-\ell})\ge A_{d-\ell-1}}}
|\prod_{i=1}^{d-\ell} \mathcal E^\gamma_\lambda \eff{d-\ell-1}i (x)|^\frac{1}{d-\ell} \Big\|_{L^q(d\mu)}  \le
C
A^{-C}_{d-\ell-1}\lambda^{-\frac{\alpha -\ell} q} \|f\|_p.
\]

By \eqref{eqn:1d-decomp}, \eqref{qqq} and these two estimates, we
get
\begin{align*}
Q_\lambda \le C \sum_{i=1}^{d-\ell-1}
A_{i-1}^{-C}{A_i}^{1-\frac1p -\frac{{\beta_{d-\ell}(\alpha-\ell)}}{q}}
Q_{\lambda} + C A^{-C}_{d-\ell-1}\lambda^{-\frac{\alpha-\ell} q}.
\end{align*}
Since $1-\frac1p-\frac{\beta_{d-\ell}(\alpha-\ell)}{q}>0$, we can
choose $A_1, \dots, A_{d-\ell-1}$, successively,   so that $
CA_{i-1}^{-C}$
${A_i}^{1-\frac1p-\frac{\beta_{d-\ell}(\alpha-\ell)}{q}}<\frac{1}{2(d-\ell)}$
for $i=1,\dots, d-\ell-1$. 
Therefore, we obtain $ Q_\lambda \le
\frac12 Q_\lambda +\lambda^{-\frac{\alpha-\ell} q}$, which implies
\eqref{qq}.

\subsection{\it Proof of Theorem \ref{thm:osc-est2}}
To prove $k$-linear estimate for $p,q$ satisfying $\frac 1k (1 - \frac 1 p )>\frac1q$ we no longer make use of Plancherel's
theorem,  but
we may still use the linear oscillatory integral estimate which is of 1-dimensional in its nature. The following is basically
interpolation between $k$-linear and linear estimates.

\begin{proposition}\label{l2fractal}
    Let $\mathcal I_1,\dots, \mathcal I_k$, $\gamma$, and $\mu$ be given as in Proposition \ref{pro:l2fractal}.
    If $\epsilon>0$ is sufficiently small, then for $p,q$ satisfying
    $q \ge k$ and
    $
    \frac 1k (1 - \frac 1 p ) \le \frac 1 q \le 1- \frac 1p ,
    $
    there is a constant $C$, independent of $\gamma$, such that
    \begin{equation}\label{eqn:k-linear}
    \| \prod_{i=1}^k \mathcal E_\lambda^\gamma f_i \|_{L^{\frac qk} (d\mu)}
    \le C \langle\mu\rangle_\alpha^{\,\,\frac kq} \lambda^{-k(1-\frac1p - \frac{d-\alpha}{q})}\prod_{i=1}^k\| f_i \|_p
    \end{equation}
    whenever $f_i$ is supported in $\mathcal I_i$, $i=1,2,\dots, k$.
\end{proposition}

\begin{proof} For $q\ge k$, Minkowski's inequality gives
    \[
    \| \prod_{i=1}^k \mathcal E_\lambda^\gamma f_i \|_{L^{\frac qk} (d\mu)} \le \|
    \prod_{i=1}^k \mathcal E_\lambda^\gamma f_i \|_{L^{\frac qk} (\mathbb R^d)} \| |
    \phi_\lambda | \ast \mu\|^{\frac k q}_{\infty} \le C \langle \mu\rangle_\alpha^{\,\,\frac kq}
    \lambda^{\frac{k(d-\alpha) }q} \| \prod_{i=1}^k \mathcal E_\lambda^\gamma f_i
    \|_{L^{\frac q k}(\mathbb R^d)}.
    \]
Thus it suffices to show that for $\frac 1k (1 - \frac 1 p ) \le
\frac 1 q \le 1- \frac 1p $,
    \begin{equation}\label{interpolation}
    \| \prod_{i=1}^k \mathcal E_\lambda^\gamma f_i
    \|_{ L^{\frac qk}(\mathbb R^d)} \le \lambda^{-k(1-\frac 1p) } \prod_{i=1}^k \| f_i \|_{L^p(I)}.
    \end{equation}

   For $\frac kq= 1-\frac1p ,\,\, 1\le p\le 2$, the estimate \eqref{interpolation} follows by 
    interpolation between the $L^2\rightarrow L^2$ estimate \eqref{eqn:k-l2} and the trivial 
    estimate $ \| \prod_{i=1}^k \mathcal E_\lambda^\gamma f_i \|_{L^\infty} \le \prod_{i=1}^k \|f_i\|_{L^1}$, provided $f_i$ is supported in $\mathcal I_i$, $i=1,2,\dots,
k$.
On the other hand,  since $|\partial_{x_1}\partial_t (x\cdot
\gamma(t))|\sim 1$, using
    H\"ormander's generalization of Hausdorff-Young's theorem, we have
  $\| \mathcal E_\lambda^\gamma f \|_{L^q}\le
  C\lambda^{-(1-\frac1p)}\|f\|_p.$
    By H\"older's inequality  we obtain  \eqref{interpolation} for $\frac 1q=  1-\frac1p ,\,\, 1\le p\le
    2$. Therefore, by  interpolating the estimates  for
    $\frac kq= 1-\frac1p$ and $\frac 1q=  1-\frac1p$ we obtain \eqref{interpolation} for $p,q$ satisfying
    $q \ge k$ and $
    \frac 1k (1 - \frac 1 p ) \le \frac 1 q \le 1- \frac 1p
    $.
\end{proof}

Once Proposition \ref{l2fractal} is obtained, one can prove Theorem
\ref{thm:osc-est2} by adopting the same line of argument in the proof of Theorem \ref{thm:osc-est}. So we shall be brief. We use
\eqref{eqn:k-linear} with $k = [d-\alpha]+1$ to estimates the
multilinear terms in \eqref{eqn:1d-decomp},  and  to the linear
terms we apply Lemma \ref{lem:induction}. Thus, we  have
\begin{equation*}
Q_\lambda \le C
\langle\mu\rangle_\alpha^{\,\,\frac1q}\sum_{i=1}^{[d-\alpha]}
A_{i-1}^{-C}{A_i}^{1 - \frac 1 p
-\frac{\beta_{[d-\alpha]+1}(1-\langle d-\alpha\rangle)}{q} }
Q_{\lambda} + C \langle\mu\rangle_\alpha^{\,\,\frac 1 q}
A^{-C}_{d-\ell-1}\lambda^{-(1 -\frac 1 p - \frac{d-\alpha}{q})}
\end{equation*}
provided that $q \ge [d-\alpha]+1$ and $(1-1/p)/([d-\alpha]+1) \le
1/q \le 1-1/p$.  Therefore,
we obtain $Q_\lambda \lesssim
\langle\mu\rangle_\alpha^{\,\,\frac1q}\lambda^{- (1 -\frac 1 p -
\frac{d-\alpha}{q})}$ whenever $1 - \frac 1 p
-\frac{\beta_{[d-\alpha]+1}(1-\langle d-\alpha\rangle)}{q} >0 $.
This completes the proof.

\section{Proof of Theorem \ref{thm:erd-ober}}\label{sec:pf bi}

Proof of Theorem \ref{thm:erd-ober} is based on an adaptation of
Erdo\u{g}an's argument in \cite{Er}. (Also see \cite{ErOb}.) The
following is basically a 2-dimensional result in that we only need to
assume $\gamma'$ and $\gamma''$ are linearly independent. To begin
with, by finite decomposition, translation and scaling  we may
assume, as before,  that $\gamma$ is close to  $\gamma_\circ^d$ such
that $\|\gamma-\gamma_\circ^d\|_{C^N(I)}\lesssim \epsilon_0$ for
sufficiently large $N$ and small enough $\epsilon_0$.

\subsection{\it Geometric observations}

To estimate the integrals on the right hand side of (\ref{eqn:whit ineqn}), we begin with some geometric observations regarding the curves.
\begin{lem}\label{lem:geo_rectangle}
  Let $\mathcal I=[\tau_{1},\tau_{2}] \subset [0,1]$  be an interval of length $L \gtrsim \lambda^{-\frac{1}{2}}$,
  then $\lambda \gamma (\mathcal I) + O(1)$ is contained in a parallelotope $\lambda  M^{\gamma,d}_{\tau_1} {R_{L}}+\lambda\gamma(\tau_1)$ where ${R_{L}}$ is a rectangle of dimension $C L \times CL^{2} \times \cdots \times C L^{2}$ which is centered at the origin.
\end{lem}

\begin{proof} To see this it is sufficient to show that  $ \gamma (\mathcal I) + O(\lambda^{-1})$ is contained in $M^{\gamma,d}_{\tau_1}{R_{L}}+\gamma(\tau_1)$.
For any $t \in [\tau_{1},\tau_{2}]$, by Taylor's expansion, we have
\begin{align*}
  \gamma(t) - \gamma(\tau_{1})
  = M^{\gamma,d}_{\tau_1}D_{L}^d{\gamma_{\circ}^{d}}(\frac{t-\tau_{1}}{L}) + \mathbf e(t,\tau_{1},L)
\end{align*}
where $\mathbf e(t,\tau_{1},L) \lesssim L^{d+1}$. So
$\gamma(\mathcal I)-\gamma(\tau_1)$ is contained in
$M^{\gamma,d}_{\tau_1} R$ where $R$ is a rectangle of dimension $\sim
L\times L^2\times \dots\times L^d$ which is centered at the origin.
Since $\lambda^{-1}\lesssim L^2$ it is clear that $ \gamma (\mathcal
I) + O(\lambda^{-1})$ contained in
$M^{\gamma,d}_{\tau_1}{R_{L}}+\gamma(\tau_1)$.
\end{proof}

The following concerns  the size of intersection of tubular neighborhoods of curves.
\begin{lem}\label{lem:geolem}
  Let ${\mathcal I}, {\mathcal J} \subset [0,1]$ be intervals satisfying $|{\mathcal I}|, |\mathcal J| \sim 2^{-n} $ and $ \dist({\mathcal I}, {\mathcal J}) \sim 2^{-n}$ with $ 2^{n} \leq \lambda^{\frac{1}{2}}$.
  Then, for $y \in \mathbb{R}^d$,
  \begin{align}\label{eqn:upboundinter}
    \big| (y + \lambda\gamma({\mathcal I}) + B(0,C)) \cap (\lambda\gamma({\mathcal J}) + B(0,C) )\big| \lesssim 2^n.
  \end{align}
 \end{lem}

\begin{proof}  As before, by a change of variables it is sufficient to show that
\[\big| (y + \gamma({\mathcal I}) + B(0,C\lambda^{-1})) \cap (\gamma({\mathcal J}) + B(0,C\lambda^{-1}) )\big| \lesssim 2^n\lambda^{-d}.\]
Let $V$ be the subspace spanned by   $\gamma'(0)$ and $\gamma''(0)$, and $P_V$ be the projection to $V$.  Since both sets are contained
in $O(\lambda^{-1})$-neighborhood of arcs, it suffices to show that
\[ \big| P_V(y + \gamma({\mathcal I}) + B(0,C\lambda^{-1})) \cap P_V(\gamma({\mathcal J}) + B(0,C\lambda^{-1}) )\big| \lesssim 2^n\lambda^{-2}.\]
The sets $P_V(y + \gamma({\mathcal I}) + B(0,C\lambda^{-1})) $, $P_V(\gamma({\mathcal J}) + B(0,C\lambda^{-1}) )$ are contained in $P_V(y) + P_V\gamma({\mathcal I}) + O(\lambda^{-1})$, $ P_V\gamma({\mathcal J}) + O(\lambda^{-1})$. Since $\gamma$  is close to $\gamma_\circ^d$, $P_V\gamma$ is close to $(t,t^2/2)$.
So,  $P_V(y) + P_V\gamma({\mathcal I}) + O(\lambda^{-1})$ , $ P_V\gamma({\mathcal J}) + O(\lambda^{-1})$ are contained in neighborhoods of curves of the form $(t,t^2/2)+O(\lambda^{-1})$ for $t\in \mathcal I$, $t\in \mathcal J$ respectively,
and the angle between them is $\sim 2^{-n}$. Hence we get the desired bound.
\end{proof}

\begin{lem}\label{lem:lengthened-support}

Let $\lambda\gtrsim \delta^{-2} $.
    Let ${\mathcal I}$ be an interval contained in $[\tau-\delta, \tau+\delta]$  and $R^*$ be the rectangle
    $R^*=\{x:  |\langle x,\frac{\gamma'(\tau)}{|\gamma'(\tau)|} \rangle|\le 1/\delta,  |\langle v_2,x \rangle| \le 1, \dots, |\langle v_d,x \rangle| \le 1  \}$
    where $v_2, \dots, v_d$ are an orthonormal basis for the orthogonal complement of span $\{\gamma'(\tau)\}$.
Then, for a {sufficiently} large $C>0$,
    \begin{align*}
        \lambda\gamma(\mathcal I)+  R^*+O(1) \subset  \lambda\gamma(\mathcal I)+ O(C)
    \end{align*}
\end{lem}

\begin{proof}
The above is equivalent to
    \begin{align*}
        \gamma(\mathcal I)+  \lambda^{-1}R^*+O(\lambda^{-1}) \subset  \gamma(\mathcal I)+ O(C\lambda^{-1}).
    \end{align*}
Any member in the left-hand side set can be written as $\gamma(t)+ s\gamma'(\tau)+
 O(\lambda^{-1}) $ for some $s$, $|s|\lesssim (\lambda\delta)^{-1}$.  Hence  we need only to show
\[ |\gamma(t)+ s\gamma'(\tau)-\gamma(t+s)|\le C\lambda^{-1}\]
    whenever $|s|\lesssim (\lambda\delta)^{-1}$. However this is clear because
    $\gamma(t)+ s\gamma'(\tau)-\gamma(t+s)=\int_{t+s}^t\gamma'(u)-\gamma'(\tau) du=O(\lambda^{-1})$.\end{proof}

Let $\varphi$ be a fixed Schwartz function which is equal to 1 in a unit cube $Q$ centered at the origin and vanishes outside $2Q$. Moreover $\widehat{\varphi}$ satisfies
\begin{align}\label{eqn:sch-decay}
  | \widehat{\varphi}(x) | \leq C_M \sum^{\infty}_{j=1} 2^{-Mj} \chi_{2^j Q} (x), \quad \textrm{for each} \,\, x \in \mathbb{R}^d,
  \, M \in \mathbb{Z}^+.
\end{align}
For a rectangle $R \subset \mathbb{R}^d$, we denote $\varphi_R$ as $\varphi\circ a_R^{-1}$, where $a_R$ is an affine mapping which takes $Q$ onto $R$.
The following lemma is a slight generalization of Lemma 3.1 in \cite{Er1}.

\begin{lem}\label{lem:uncertain} Let $\lambda_d \leq \cdots \leq \lambda_2 \leq \lambda_1 \lesssim \lambda$ and let $\mu$ be a positive Borel measure supported in $B(0,1)$ satisfying \emph{(\ref{eqn:alpha-diml measure})}.
  Let $R$ be a rectangle of dimensions $\lambda_1 \times \lambda_2 \times \cdots \times \lambda_d$,
  $R^*$ be the dual set of $R$ centered at the origin, and $A$ be a nonsingular matrix.
  Then \\
  \emph{(i)} $\| \mu \ast |\mathcal F(\varphi_R\circ A^{-1})| \|_{\infty} \lesssim \langle\mu\rangle_\alpha |\det A| \|A^{-t}\|^{\alpha} \lambda_{1}^{d-\alpha}$, \\
  \emph{(ii)} $\int_{K A^{-t} R^*} \mu \ast |\mathcal F(\varphi_R\circ A^{-1})| ( x + y ) dy \lesssim \langle\mu\rangle_\alpha K^{\alpha}
  \|A^{-t}\|^{\alpha} \lambda_{1}^{d-\alpha} \prod_{k=1}^{d}\lambda_k^{-1}$,
  for $K \gtrsim 1$ and $x \in \mathbb{R}^d$.
\end{lem}

\begin{proof}
Fixing a large enough $M$, by  \eqref{eqn:sch-decay} and  change of variables we have
    \begin{align*}
    \begin{aligned}
    | \widehat{\varphi_R}(x) |
    & \leq C_M \prod_{k=1}^d \lambda_k \sum^{\infty}_{j=1} 2^{-Mj} \chi_{2^j R^*} (x),
    \end{aligned}
    \end{align*}
    which gives
    \begin{align}\label{eqn:schdecay}
    \int |\widehat{\varphi_{R}}(A^{t}(x-y))| d\mu(y)
    & \lesssim \prod_{k=1}^d \lambda_k \sum^{\infty}_{j=1} 2^{-Mj} \int \chi_{2^j R^*} (A^{t}(x-y)) d\mu(y).
    \end{align}
Since $2^{j}R^*$ is covered by as many as $\sim\lambda_{1}^{d}/\prod_{k=1}^d \lambda_k$ cubes with side-length $2^{j}\lambda^{-1}_{1}$, 
applying  (\ref{eqn:alpha-diml measure}) and \eqref{munorm} to each of the cubes gives
   $
    \int \chi_{2^j R^*} (A^{t}(x-y)) d\mu(y) \lesssim \langle\mu\rangle_\alpha (\|A^{-t}\| 2^{j} \lambda_{1}^{-1})^{\alpha} \frac{\lambda_{1}^{d}}{\prod_{k=1}^d \lambda_k}.
    $  
  Since $\mu \ast |\mathcal F(\varphi_{R}\circ A^{-1})|(x) = |\det A| \int |\widehat{\varphi_{R}}(A^{t}(x-y))| d\mu(y)$,
  by combining the inequalities we get
    \begin{align*}
    \mu \ast |\mathcal F(\varphi_{R}\circ A^{-1})|(x)
    \lesssim |\det A|  \sum^{\infty}_{j=1} 2^{-Mj} \langle\mu\rangle_\alpha (\|A^{-t}\|2^{j}\lambda_{1}^{-1})^{\alpha} {\lambda_{1}^{d}}
    \lesssim \langle\mu\rangle_\alpha |\det A| \|A^{-t}\|^{\alpha} \lambda^{d-\alpha}_{1}.
    \end{align*}
 This proves (i).

We now turn to (ii). Without loss of generality we may assume $ x = 0 $.  By \eqref{eqn:schdecay},
    \begin{align}\label{eqn:uc31}
   \begin{aligned}
   & \int_{K A^{-t} R^*} \mu \ast |\mathcal F(\varphi_{R}\circ A^{-1})|(y) dy
    \\ \lesssim |\det A| &\prod^d_{k=1}\lambda_{k} \sum^{\infty}_{j=1} 2^{-Mj} \iint \chi_{K A^{-t} R^*}(y) \chi_{2^{j}A^{-t}R^*}(y-u) d\mu(u) dy. 
    \end{aligned}
    \end{align}
Since $u-y \in 2^{j} A^{-t} R^*$ in the last integrand,
    $
     \chi_{K A^{-t} R^*}(y)
    \lesssim  \chi_{(K+2^{j}) A^{-t} R^*}(u)$.  So we have
    \begin{align*}
     \iint \chi_{K A^{-t} R^*}(y) &\chi_{2^{j}A^{-t}R^*}(y-u) d\mu(u)dy
     \le  |\det A^{-t}|\frac{2^{jd}}{\prod^d_{k=1}\lambda_{k}}
    \int \chi_{(K+2^{j}) A^{-t} R^*}(u)  d\mu(u)  \\
    &\lesssim  |\det A^{-t}|\frac{2^{jd}}{\prod^d_{k=1}\lambda_{k}} \langle\mu\rangle_\alpha (\|A^{-t}\|(K+2^{j})\lambda_{1}^{-1})^{\alpha} \frac{\lambda_{1}^{d}}{\prod^d_{k=1}\lambda_{k}}.
    \end{align*}
  For the last inequality   we cover  $(K+2^{j})R^*$  with $O(\frac{\lambda_{1}^{d}}{\prod^d_{k=1}\lambda_{k}})$ cubes of side length $(K+2^{j})\lambda_{1}^{-1}$ and use \eqref{eqn:alpha-diml measure} and \eqref{munorm}.
By combining this and \eqref{eqn:uc31}, we  get (ii).
\end{proof}

\subsection{\it Whitney type decomposition} By a Whitney type decomposition  we may write
\begin{align}\label{eqn:whit-set}
    [0,1] \times [0,1] = \bigg[ \bigcup_{4 \leq 2^n \leq \lambda^{\frac{1}{2}}} \Big[ \bigcup_{\substack{|\mathcal I|=|\mathcal J|=2^{-n} \\ \mathcal I \sim \mathcal J}} (\mathcal I \times \mathcal J) \Big] \bigg] \bigcup D
\end{align}
where $\mathcal I$, $\mathcal J$ are dyadic intervals, $D$ is  a union of finitely overlapping boxes of side length $\approx \lambda^{-\frac{1}{2}}$, and $D$ is contained in the $C\lambda^{-\frac{1}{2}}$-neighborhood of the diagonal $\{ (x,x): x \in [0,1]\}$.
Here, we say $\mathcal I \sim \mathcal J$ to mean that $\mathcal I$, $\mathcal J$ are not adjacent but have adjacent parent intervals.

For  $\mathcal I=[a, b]$ we set
\begin{align*} \Omega_{ \lambda, \mathcal I } =
\begin{cases}
\,\, \{ x \in \lambda \gamma(I) + O(1) :
\gamma'(b)\cdot {(x-\gamma(b))}\le  0\le \gamma'(a) \cdot {(x-\gamma(a))} \}  
&\text{ if  $a\neq 0$, $b\neq 1$},  \\
\,\, \{ x \in \lambda \gamma(I) + O(1) :
\gamma'(b)\cdot {(x-\gamma(b))}\le  0\}  &\text{ if  $a=0$},\\
\,\,  \{ x \in \lambda \gamma(I) + O(1) :      0\le \gamma'(a) \cdot {(x-\gamma(a))} \}  &\text{ if  $b= 1$},
  \end{cases}
  \end{align*}
and set
\begin{align}\label{def:omega}
    g_{\mathcal I} = g \cdot \chi_{\Omega_{\lambda,\mathcal I}}.
\end{align}
For distinct dyadic intervals $\mathcal I, \mathcal J \subset [0,1]$, the intersection of $\Omega_{\lambda, \mathcal I}$ and $\Omega_{\lambda, \mathcal J}$ has Lebesgue measure zero in $\mathbb{R}^{d}$ because $2^{-n}\ge \lambda^{-1/2}$.
This leads to
\begin{align*}
    | \widehat{g}(x) |^2 \leq \sum_{n \geq 2}^{\log\lambda^{\frac{1}{2}}} \sum_{\substack{|\mathcal I|=|\mathcal J|=2^{-n} \\ \mathcal I \sim \mathcal J}} | \widehat{g_{\mathcal I}}(x) \widehat{g_{\mathcal J}}(x) | + 2\sum_{\mathcal I \in \mathfrak I_{E}} | \widehat{g_{\mathcal I}}(x) |^{2}
\end{align*}
where $\mathfrak I_{E}$ is a finitely overlapping set of dyadic intervals $\mathcal I$ with $|\mathcal I| \approx \lambda^{-\frac{1}{2}}$. Using above inequality, we have for any $q \geq 2$,
\begin{align}\label{eqn:whit ineqn}
    \| \widehat{g} \|^2_{L^q(d\mu)} \leq \sum_{n \geq 2}^{\log\lambda^{\frac{1}{2}}} \sum_{\substack{|\mathcal I|=|\mathcal J|=2^{-n} \\ \mathcal I \sim \mathcal J}} \| \widehat{g_{\mathcal I}} \widehat{g_{\mathcal J}} \|_{L^{\frac{q}{2}}(d\mu)} + \sum_{\mathcal I \in \mathfrak I_{E}} \| \widehat{g_{\mathcal I}} \|^2_{L^q(d\mu)}.
\end{align}

\subsubsection{Estimate for $g_{\mathcal I}$, $\mathcal I \in \mathfrak I_E$} 
For $\mathcal I = [\tau_1,\tau_2] \in \mathfrak I_E$, we have $2^{-n} \approx \lambda^{-1/2}$. By Lemma \ref{lem:geo_rectangle}  the support of $g_{\mathcal I}$, i.e.
$\Omega_{\lambda,\mathcal I}$ is contained in a parallelotope $M^{\gamma,d}_{\tau_{1}}R$ where $R$ is a rectangle of dimensions $C\lambda^{\frac{1}{2}} \times C \times \cdots \times C$ .
Hence $\widehat{g_{\mathcal I}} = \widehat{g_{\mathcal I}} \ast \mathcal F(\varphi_R \circ (M^{\gamma,d}_{\tau_{1}})^{-1})$. 
Since
$\| \mathcal F(\varphi_R \circ (M^{\gamma,d}_{\tau_{1}})^{-1}) \|_{1}\lesssim C$,   by H\"older's inequality we get
\begin{align*}
  | \widehat{g_{\mathcal I}} |
  \lesssim ( | \widehat{g_{\mathcal I}} |^q \ast | \mathcal F(\varphi_R \circ (M^{\gamma,d}_{\tau_{1}})^{-1}) | )^{\frac{1}{q}}.
\end{align*}
So, we have
\begin{align}\label{eqn:linear_decay_est}
  \begin{aligned}
    \| \widehat{g_{\mathcal I}} \|^{q}_{L^{q}(d\mu)}
    \lesssim \int (|\widehat{g_{\mathcal I}}|^{q} \ast | \mathcal F(\varphi_R \circ (M^{\gamma,d}_{\tau_{1}})^{-1}) |)(x) d\mu(x)
    \lesssim \langle\mu\rangle_\alpha \lambda^{\frac{d-\alpha}{2}} \| \widehat{g_{\mathcal I}} \|^{q}_{q}.
  \end{aligned}
\end{align}
The last inequality follows from (i) in Lemma \ref{lem:uncertain} and the fact that $R$ has dimensions $C\lambda^{\frac{1}{2}} \times C \times \cdots \times C$. Since $q \geq 2$,  by Hausdorff-Young inequality and H\"older's inequality, we have
\begin{align*}
 \|\widehat {g_{\mathcal I}}\|_{q} \le \|g_{\mathcal I}\|_{q'}
  \leq  \|g_{\mathcal I}\|_{2}
 |\Omega_{\lambda, \mathcal I}|^{(\frac12 -\frac1q)}
  \lesssim \| g_{\mathcal I} \|_{2} (\lambda^{\frac{1}{2}})^{\frac{1}{2}-\frac{1}{q}}.
\end{align*}
Thus, combining this with \eqref{eqn:linear_decay_est},
\begin{align}\label{IE}
  \sum_{\mathcal I \in \mathfrak I_E} \| \widehat{g_{\mathcal I}} \|^2_{L^q(d\mu)}
\lesssim \langle\mu\rangle_\alpha^{\,\,\frac2q}\lambda^{\frac{d-\alpha}{q}} \lambda^{\frac{1}{2}-\frac{1}{q}} \sum_{\mathcal I \in \mathfrak I_E} \|g_{\mathcal I}\|^2_2
  \lesssim \langle\mu\rangle_\alpha^{\,\,\frac2q}\lambda^{\frac{1}{2} + \frac{(d-1)-\alpha}{q}} \|g\|^2_2.
\end{align}

\subsubsection{Bilinear term estimate}
Firstly, we assume $q=2$. Fix $n$ with $4 \leq 2^n \leq \lambda^{1/2}$ and a pair $\mathcal I = [\tau_{1},\tau_{2}], \mathcal J = [\tau_{3},\tau_{4}]$ of dyadic intervals with $|\mathcal I| = |\mathcal J| = 2^{-n}$ and $\mathcal I \sim \mathcal J$. 
Since $\mathcal I \sim \mathcal J$, the support of $ g_{\mathcal I} \ast g_{\mathcal J} $ is contained in a parallelotope $M_{\tau_{1}}^{\gamma,d}R$ where $R$ is a rectangle with dimensions $ 2C\lambda 2^{-n} \times 2C\lambda 2^{-2n} \times \cdots \times 2C\lambda 2^{-2n} $.
Using $ g_{\mathcal I} \ast g_{\mathcal J} = ( g_{\mathcal I} \ast g_{\mathcal J} ) (\varphi_{R}\circ (M_{\tau_{1}}^{\gamma,d})^{-1}) $, we obtain
\begin{align}\label{ap1}
  \int | \widehat{g_{\mathcal I}}(x) \widehat{g_{\mathcal J}}(x) | d\mu(x)
  \lesssim \int | \widehat{g_{\mathcal I}}(x) \widehat{g_{\mathcal J}}(x) | ( \mu \ast | \mathcal F( \varphi_R \circ (M^{\gamma,d}_{\tau_{1}})^{-1} ) | )(x) dx.
\end{align}

Consider a tiling of $\mathbb{R}^d$ with rectangles $T$ of dimensions $C2^{-n} \times C \times \cdots \times C$. 
Note that each $T$ is contained in a rectangle $x_{T} + C\lambda 2^{-2n}R^*$ for some $x_{T} \in \mathbb{R}^d$.
Also let $\phi$ be a fixed non-negative Schwartz function satisfying $\phi > 1/2$ on $Q$, $ \supp \widehat{\phi} \subseteq Q$ and the inequality of (\ref{eqn:sch-decay}).
Using the properties of $\phi$, we obtain
$  1 \lesssim \sum_T \phi^3_T \lesssim \sum_T \phi^2_T \lesssim 1$,
where $\phi_T := \phi \circ a^{-1}_T$.

Set $\widehat{g_{\mathcal I, T}} :=  \widehat{g_{\mathcal I}} \cdot(\phi_{T}\circ(M^{\gamma,d}_{\tau_{1}})^{t}).$   
By $1 \lesssim \sum_T \phi^3_T$ and  Cauchy-Schwarz inequality, we get
\begin{equation}\label{ap2}
\begin{aligned}
  & \int | \widehat{g_{\mathcal I}}(x) \widehat{g_{\mathcal J}}(x) | (\mu \ast | \mathcal F( \varphi_R \circ (M^{\gamma,d}_{\tau_{1}})^{-1}) |)(x) dx  \\
  & \lesssim \sum_{T} \int | \widehat{g_{\mathcal I, T}}(x) \widehat{g_{\mathcal J, T}}(x) | (\mu \ast | \mathcal F( \varphi_R \circ (M^{\gamma,d}_{\tau_{1}})^{-1}) |)(x) (\phi_{T}\circ(M^{\gamma,d}_{\tau_{1}})^{t}) (x) dx  \\
  & \lesssim \sum_{T} \|  \widehat{g_{\mathcal I, T}} \widehat{g_{\mathcal J, T}} \|_2 \,\|\mu \ast | \mathcal F( \varphi_R \circ (M^{\gamma,d}_{\tau_{1}})^{-1}) | (\phi_{T}\circ(M^{\gamma,d}_{\tau_{1}})^{t}) \|_2.
\end{aligned}
\end{equation}

By a standard argument
\begin{align*}
  \begin{aligned}
    &\int | \widehat{g_{\mathcal I, T}}(x) \widehat{g_{\mathcal J, T}}(x)|^2 dx = \int | \widetilde{g_{\mathcal I, T}} \ast \overline{g_{\mathcal J, T}}(y) |^2 dy \\
    &\hspace {15 mm} \leq \sup_{y} | (y + \supp(g_{\mathcal I, T})) \cap \supp(g_{\mathcal J, T}) | \| g_{\mathcal I, T} \|^2_2 \| g_{\mathcal J, T} \|^2_2.
  \end{aligned}
\end{align*}

By Lemma \ref{lem:lengthened-support}, $y + \supp(g_{\mathcal I, T})$, $\supp(g_{\mathcal J, T})$ are contained
in $y + \lambda\gamma({\mathcal I}) + B(0,C)$, $\lambda\gamma({\mathcal J}) + B(0,C)$, respectively.
Thus, Lemma \ref{lem:geolem} implies $ \sup_{y} | (y + \supp(g_{\mathcal I, T})) \cap \supp(g_{\mathcal J, T}) | \lesssim 2^n$.
So, we get
\begin{align}\label{eqn:par-cs-young ineq}
  \int | \widehat{g_{\mathcal I, T}}(x) \widehat{g_{\mathcal J, T}}(x)|^2 dx \lesssim 2^n \| g_{\mathcal I, T} \|^2_2 \| g_{\mathcal J, T} \|^2_2.
\end{align}

Now we show
\begin{equation}
\label{ap3} \|\mu \ast | \mathcal F( \varphi_R \circ
(M^{\gamma,d}_{\tau_{1}})^{-1}) |
(\phi_{T}\circ(M^{\gamma,d}_{\tau_{1}})^{t}) \|_2^2
  \lesssim \langle\mu\rangle_\alpha^{\,\,2} \lambda^{d-\alpha} 2^{-n}.
\end{equation}
First we note that by (i) in Lemma \ref{lem:uncertain},
\begin{align}\label{eqn:fractal-decay}
  \| \mu \ast | \mathcal F( \varphi_R \circ (M^{\gamma,d}_{\tau_{1}})^{-1}) | \|_{\infty} \lesssim \langle\mu\rangle_\alpha \lambda^{ d - \alpha } (2^{-n})^{ d - \alpha }.
\end{align}
Using  
\eqref{eqn:sch-decay} for $\phi_T$ and (ii) in Lemma \ref{lem:uncertain} with recalling that $T$ is contained in $x_T + C \lambda 2^{-2n} R^*$ for some $x_T \in \mathbb{R}^d$, we have
\begin{align*}
  & \int (\mu \ast | \mathcal F( \varphi_R \circ (M^{\gamma,d}_{\tau_{1}})^{-1}) |)(x) (\phi_{T}\circ(M^{\gamma,d}_{\tau_{1}})^{t}) (x) dx  \\
&  \lesssim  \sum^{\infty}_{j=1} 2^{-Mj} \int_{2^{j}\lambda 2^{-2n} (M^{\gamma,d}_{\tau_{1}})^{-t} R^*} \mu \ast | \mathcal F( \varphi_R \circ (M^{\gamma,d}_{\tau_{1}})^{-1}) |(x - 2^j (M^{\gamma,d}_{\tau_{1}})^{-t}x_{T}) dx \\
&  \lesssim \langle\mu\rangle_\alpha (2^{n})^{d-1-\alpha}.
\end{align*}
Since $\phi_T (x) \lesssim 1$, by combining this and \eqref{eqn:fractal-decay}, we get \eqref{ap3}.

By the inequalities  \eqref{ap1}, \eqref{ap2}, \eqref{eqn:par-cs-young ineq}, and \eqref{ap3} and using the fact that $\sum_{T} \phi^{2}_{T} \lesssim 1$
\begin{align}\label{ap4}
  \| \widehat{g_{\mathcal I}} \widehat{g_{\mathcal J}} \|_{L^1(d\mu)} \lesssim \langle\mu\rangle_\alpha \lambda^{\frac{d-\alpha}{2}} \sum_T \| g_{\mathcal I, T} \|_2 \| g_{\mathcal J, T} \|_2
  \lesssim \langle\mu\rangle_\alpha \lambda^{\frac{d-\alpha}{2}} \| g_{\mathcal I} \|_2 \| g_{\mathcal J} \|_2.
\end{align}
For the last inequality, we used the Cauchy-Schwarz inequality and  Plancherel's theorem.

By \eqref{ap4}, we have
\begin{align*}
  \sum_{ 4 \leq 2^n \leq \lambda^{1/2} } \sum_{ \substack{|\mathcal I| = |\mathcal J| = 2^{-n} \\ \mathcal I \sim \mathcal J }} \| \widehat{g_{\mathcal I}} \widehat{g_{\mathcal J}} \|_{L^1(d\mu)}
   \lesssim \langle\mu\rangle_\alpha \lambda^{\frac{d-\alpha}{2}} \sum_{ 4 \leq 2^n \leq \lambda^{1/2} } \sum_{|\mathcal I|=2^{-n}} \sum_{\mathcal I \sim \mathcal J} \| g_{\mathcal I} \|_2 \| g_{\mathcal J} \|_2  \\
    \lesssim \langle\mu\rangle_\alpha \lambda^{\frac{d-\alpha}{2}} \sum_{ 4 \leq 2^n \leq \lambda^{1/2} } \Big( \sum_{|\mathcal I|=2^{-n}} \|g_{\mathcal I}\|^2_2 \Big)^{\frac{1}{2}} \Big( \sum_{|\mathcal I|=2^{-n}} \|g_{\mathcal I}\|^2_2 \Big)^{\frac{1}{2}}
  \lesssim \langle\mu\rangle_\alpha \lambda^{\frac{d-\alpha}{2} }\log \lambda  \| g \|_2^2.
\end{align*}
For the second inequality we use the fact that  there are finitely many intervals $\mathcal J$ related to $\mathcal I$ for each dyadic interval $\mathcal I$.  Thus we get the required bound in the case $q=2$.

Now we assume $q \geq 4$. Let $\mathcal I$, $\mathcal J$ with $\mathcal I\sim \mathcal J$,  and $R$ be defined as before. Using $ g_{\mathcal I} \ast g_{\mathcal J} = ( g_{\mathcal I} \ast g_{\mathcal J} ) (\varphi_{R}\circ (M_{\tau_{1}}^{\gamma,d})^{-1}) $,  H\"older's inequality, and (\ref{eqn:fractal-decay}), we have
\begin{align*}
  \|\widehat{g_{\mathcal I}}\widehat{g_{\mathcal J}}\|^{\frac{q}{2}}_{L^{\frac{q}{2}}(d\mu)}
  & \lesssim \langle\mu\rangle_\alpha \lambda^{ d - \alpha } (2^{-n})^{ d - \alpha } \|\widehat{g_{\mathcal I}}\widehat{g_{\mathcal J}}\|^{\frac{q}{2}-2}_{\infty} \int |\widehat{g_{\mathcal I}}(x)\widehat{g_{\mathcal J}}(x)|^2 dx.
\end{align*}
Repeating the argument for  (\ref{eqn:par-cs-young ineq}) and using Lemma \ref{lem:geolem}, we have
$
\int |\widehat{g_{\mathcal I}}(x)\widehat{g_{\mathcal J}}(x)|^2 dx
\lesssim 2^n \|g_{\mathcal I}\|^2_2 \|g_{\mathcal J}\|^2_2.
$
Also,   by Young's inequality and Cauchy-Schwarz inequality,
$\|\widehat{g_{\mathcal I}} \widehat{g_{\mathcal J}}\|_{\infty}
\lesssim \lambda 2^{-n} \|g_{\mathcal I}\|_2 \|g_{\mathcal J}\|_2$.
Hence, we get
 \begin{align*}
\|\widehat{g_{\mathcal I}}\widehat{g_{\mathcal J}}\|^{\frac{q}{2}}_{L^{\frac{q}{2}}(d\mu)}
\lesssim \langle\mu\rangle_\alpha \lambda^{d-\alpha+\frac{q}{2}-2} (2^{n})^{-d+\alpha+3-\frac{q}{2}}\|g_{\mathcal I}\|^{\frac{q}{2}}_2 \|g_{\mathcal J}\|^{\frac{q}{2}}_2.
\end{align*}
Here, 
if $-d + \alpha + 3 - \frac{q}{2} \geq 0 $, since $2^n \leq \lambda^{\frac{1}{2}}$, then $\lambda^{d-\alpha+\frac{q}{2}-2} (2^{n})^{-d+\alpha+3-\frac{q}{2}} \leq \lambda^{\frac{q}{4}+\frac{(d-1)-\alpha}{2}}$. Otherwise, $\lambda^{d-\alpha+\frac{q}{2}-2} (2^{n})^{-d+\alpha+3-\frac{q}{2}} < \lambda^{d-\alpha+\frac{q}{2}-2}$. Hence
\begin{align*}
\|\widehat{g_{\mathcal I}}\widehat{g_{\mathcal J}}\|_{L^{\frac{q}{2}}(d\mu)}
\lesssim \langle\mu\rangle_\alpha^{\,\,\frac2q} \lambda^{\max (\frac{1}{2}+\frac{(d-1)-\alpha}{q}, 1-\frac{2\alpha}{q}+\frac{2(d-2)}{q})} \|g_{\mathcal I}\|_2 \|g_{\mathcal J}\|_2.
\end{align*}
Thus by the same argument as before, we sum along $n$, $\mathcal I$, $\mathcal J$ to get
\begin{align*}
\sum_{ 4 \leq 2^n \leq \lambda^{1/2} } \sum_{ \substack{|\mathcal I| = |\mathcal J| = 2^{-n} \\ \mathcal I \sim \mathcal J }} \|\widehat{g_{\mathcal I}}\widehat{g_{\mathcal J}}\|_{L^{\frac{q}{2}}(d\mu)}
\lesssim \langle\mu\rangle_\alpha^{\,\,\frac2q} \lambda^{\max (\frac{1}{2}+\frac{(d-1)-\alpha}{q}, 1-\frac{2\alpha}{q}+\frac{2(d-2)}{q}) + \epsilon} \| g \|_2^2.
\end{align*}
Since the intermediate cases follow by interpolation, this completes the proof.
\qed


\end{document}